\documentclass[11pt, letterpaper, reqno]{amsart}
\usepackage[utf8]{inputenc} 	% Required For including letters with accents
\usepackage{microtype} 			% Tweak font spacing
\usepackage{geometry} 			% Set margins
\usepackage{amsmath}			% ams packages
\usepackage{amsthm}		 		% amsthm environments
\usepackage{amssymb}	 		
\usepackage{bm}					% Bold symbols
\usepackage{mathrsfs}			% Script Font
\usepackage{thmtools}
\usepackage{xcolor}				% Define colors
\usepackage{tikz}				% Vector drawing
\usepackage[all]{xy}			% Commutative diagrams
\usepackage{graphicx}			% Graphics package
\usepackage{enumitem} 			% Better enumerate
\usepackage{array}
\usepackage{lmodern}			% Latin Modern Roman
\usepackage[T1]{fontenc}		% Use 8-bit encoding
\usepackage[breaklinks=true, colorlinks=true, citecolor=cyan, urlcolor=magenta]{hyperref}
\usepackage{cleveref}
\usepackage[backend=biber, style=alphabetic, backref=true, firstinits=true, isbn=false, date=year]{biblatex}
\makeatletter
\renewbibmacro{in:}{}
\DeclareFieldFormat{pages}{#1}
\renewcommand*{\bibnamedash}{%
	\leavevmode\raise +0.6ex\hbox to 5.5ex{\hrulefill}.\space\space}

\InitializeBibliographyStyle{\global\undef\bbx@lasthash}

\newbibmacro*{bbx:savehash}{\savefield{fullhash}{\bbx@lasthash}}

\renewbibmacro*{author}{%
	\ifboolexpr{
		test \ifuseauthor
		and
		not test {\ifnameundef{author}}
	}
	{%
		\iffieldequals{fullhash}{\bbx@lasthash}
		{\bibnamedash\addcomma\space}
		{\printnames{author}}%
		\usebibmacro{bbx:savehash}%  
		\iffieldundef{authortype}
		{}
		{%
			\setunit{\addcomma\space}%
			\usebibmacro{authorstrg}%
		}%
	}
	{\global\undef\bbx@lasthash}%
}
\makeatother
\newtheorem{propositionx}{Proposition}[section]
\newenvironment{proposition}
{\pushQED{\qed}\propositionx}
{\popQED\endpropositionx}
\newenvironment{propositionp}
{\pushQED{\qed}\propositionx}
{\popQED\endpropositionx}
\newtheorem{theorem}[propositionx]{Theorem}
\newtheorem{corollaryx}[propositionx]{Corollary}
\newenvironment{corollary}
{\pushQED{\qed}\corollaryx}
{\popQED\endcorollaryx}
\newtheorem{lemmax}[propositionx]{Lemma}

\newenvironment{lemmap}
{\pushQED{\qed}\lemmax}
{\popQED\endlemmax}

\theoremstyle{remark}
\newtheorem{remark}[propositionx]{Remark}
\newtheorem*{remark*}{Remark}
\newcommand{\bbB}{\mathbb{B}}
\newcommand{\bbC}{\mathbb{C}}

\newcommand{\bbN}{\mathbb{N}}

\newcommand{\bbR}{\mathbb{R}}
\newcommand{\bbS}{\mathbb{S}}
\newcommand{\bbT}{\mathbb{T}}

\newcommand{\bbZ}{\mathbb{Z}}

\newcommand{\calD}{\mathcal{D}}

\newcommand{\calF}{\mathcal{F}}

\newcommand{\calJ}{\mathcal{J}}

\newcommand{\calS}{\mathcal{S}}

\newcommand{\frakM}{\mathfrak{M}}

\newcommand{\frakV}{\mathfrak{V}}

\newcommand{\bfA}{\mathbf{A}}

\newcommand{\bfk}{\mathbf{k}}

\newcommand{\bfw}{\mathbf{w}}

\newcommand{\bmtheta}{\bm{\theta}}

\newcommand{\bmSigma}{\bm{\Sigma}}

\newcommand{\dd}{\,\mathrm{d}}

\title{Half-Waves and Spectral Riesz Means on the 3-Torus}
\author{Elliott Fairchild and Ethan Sussman*}
\date{September 28th, 2022 (Last Update), September 22nd, 2021 (Original)}
\email{ethanws@mit.edu, elliottfairchild@college.harvard.edu}
\address{Department of Mathematics, Massachusetts Institute of Technology, Massachusetts 02139-4307, USA}
\subjclass[2020]{35P20, 11Lxx, 42axx}

\begin{document}

\begin{abstract}
	For a full rank lattice $\Lambda \subset \bbR^d$ and $\bfA \in \bbR^d$, consider $N_{d,0;\Lambda,\bfA}(\Sigma) = \# ([\Lambda+\bfA] \cap \Sigma \bbB^d) = \# \{\bfk \in \Lambda : |\bfk+\bfA| \leq \Sigma \}$. Consider the iterated integrals 
	\[
	N_{d,k+1;\Lambda,\bfA}(\Sigma) = \int_0^\Sigma N_{d,k;\Lambda,\bfA}(\sigma) \dd \sigma, 
	\]
	for $k\in \bbN$. 
	After an elementary derivation via the Poisson summation formula of the sharp large-$\Sigma$ asymptotics of $N_{3,k;\Lambda,\bfA}(\Sigma)$ for $k\geq 2$ (these having an $O(\Sigma)$ error term), we discuss how they are encoded in the structure of the Fourier transform $\calF N_{3;\Lambda,\bfA}(\tau)$.
	The analysis is related to H\"ormander's analysis of spectral Riesz means, as the iterated integrals above are weighted spectral Riesz means for the simplest magnetic Schr\"odinger operator on the flat $3$-torus. That the $N_{3,k;\Lambda,\bfA}(\Sigma)$ obey an asymptotic expansion to $O(\Sigma^2)$ is a special case of a general result holding for all magnetic Schr\"odinger operators on all manifolds, and the subleading polynomial corrections can be identified in terms of the Laurent series of the half-wave trace at $\tau=0$. The improvement to $O(\Sigma)$ for $k\geq 2$ follows from a bound on the growth rate of the half-wave trace at late times. 
\end{abstract}

\maketitle

\tableofcontents

\section{Introduction}

For $d\in \bbN^+$, let $\Lambda\subset \bbR^d$ denote a full rank lattice, and let $\operatorname{Covol}(\Lambda)>0$ denote its covolume. Let $\bfA\in \bbR^d$. Consider the function $N_{d;\Lambda,\bfA} : [0,\infty)\to \bbN^+$ given by 
\begin{equation}
	N_{d;\Lambda,\bfA}(\Sigma) = \# ([\Lambda+\bfA] \cap \Sigma \bbB^d) = \# \{\bfk \in \Lambda : |\bfk+\bfA| \leq \Sigma \} ,
	\label{eq:Ndef}
\end{equation}
where $\bbB^d \subset \bbR^d$ is the closed unit ball. (For $\Sigma<0$, we define $N_{d;\Lambda,\bfA}(\Sigma)=0$.)
It is geometrically clear that $\smash{N_{d;\Lambda,\bfA}(\Sigma)} = \Sigma^d \operatorname{Vol} \bbB^d \operatorname{Covol}(\Lambda)^{-1} + O(\Sigma^{d-1})$, as Gauss observed. 
Let $\smash{N_{d;\Lambda} = N_{d;\Lambda,\bf0}}$. 
The \emph{Gauss $d$-ellipsoid problem} ($d$-\emph{sphere} problem when $\Lambda=\bbZ^d$) is to find -- in the form of a polynomial bound (possibly with logarithmic corrections) -- the size of the error $\smash{N_{d;\Lambda}(\Sigma)} - \Sigma^d \operatorname{Vol} \bbB^d \operatorname{Covol}(\Lambda)^{-1} $. The $d=2$ case is known simply as the Gauss ellipse (or circle) problem, and the $d=3$ case is known as the Gauss ellipsoid (or sphere) problem. For $d=2$, it is conjectured that 
\begin{equation}
	N_{2;\Lambda}(\Sigma) = \pi \operatorname{Covol}(\Lambda)^{-1} \Sigma^2 + O( \Sigma^\nu)
	\label{eq:n2c}
\end{equation}
if $\nu > 1/2$. See \cite{Ivic} for a bibliography.  For $d=3$, it is conjectured that 
that 
\begin{equation}
	N_{3;\Lambda}(\Sigma) = (4\pi/3) \operatorname{Covol}(\Lambda)^{-1} \Sigma^3 + O( \Sigma^{\nu})
	\label{eq:n3c}
\end{equation}
if $\nu > 1$. 
Progress can be found in a bevy of works, going back to Hardy, Landau, Ramanujan, Sierpi\'nsky, and van der Corput, allowing any $\nu>3/2$. At present, the best result for the sphere problem is due to Heath-Brown \cite{HeathBrown}, who shows that we may take any $\nu >21/16 =1.3125$ (and whose proof shows that if the generalized Riemann hypothesis holds then, assuming extraneous errors in the proof are dealt with, we may take any $\nu > 5/4$). This built off of previous work of Chamizo and Iwaniec \cite{Chamizo1995}, who showed that \cref{eq:n3c} holds for any $\nu > 29/22 = 1.31818\cdots$, which in turn was based off of papers of Chen \cite{Chen} and Vinogradov \cite{Vinogradov} which allowed any $\nu > 4/3$. See \cite{Lax}\cite{Huxley}\cite{Ivic} for related results. For arbitrary $\Lambda\subset \bbR^3$, Kr\"atzel and Nowak \cite{KratzelNowak1}\cite{KratzelNowak2} established \cref{eq:n3c} for $\nu>74/50$, this being improved by M\"uller \cite{Muller} to $\nu>63/43$ and recently Guo \cite{Guo} to $\nu>231/158\approx 1.46202\cdots$ (these results holding for more general convex bodies as well). See \cite{BentkusGotze}\cite{Gotze} for sharp results for $d\geq 5$. Estimating $N_{3;\Lambda,\bfA}$ is a difficult problem. In this note, we do something much easier: examine some related  quantities which, in contrast to $N_{3,\Lambda,\bfA}$, can be estimated sharply.

Let $N_{d,0;\Lambda,\bfA}=N_{d;\Lambda,\bfA}$ and define, for each $k\in \bbN^+$, $N_{d,k;\Lambda,\bfA}(\Sigma) = \int_0^\Sigma N_{d,k-1;\Lambda,\bfA}(\sigma)\dd \sigma $, i.e. 
\begin{align}
	\begin{split} 
	N_{d,1;\Lambda,\bfA}(\Sigma)&=\int_0^\Sigma N_{d;\Lambda,\bfA}(\sigma) \dd \sigma ,\\ N_{d,2;\Lambda,\bfA}(\Sigma)&=\int_0^\Sigma \int_0^{\sigma} N_{d;\Lambda,\bfA}(\sigma') \dd \sigma' \dd \sigma,\\  
	N_{d,3;\Lambda,\bfA}(\Sigma)&=\int_0^\Sigma \int_0^\sigma \int_0^{\sigma'} N_{d;\Lambda,\bfA}(\sigma'') \dd \sigma'' \dd \sigma' \dd \sigma , \cdots .
	\end{split}
\end{align}
(For $\Sigma\leq 0$, $N_{d;k;\Lambda,\bfA}(\Sigma)=0$.)
We will refer to these functions as ``iterated integrals'' of $\smash{N_{d;\Lambda,\bfA}}$. 
We focus on the case $d=3$. 
Even though it is somewhat straightforward (perhaps modulo relatively unimportant technicalities) to give a direct treatment, the authors do not know of one in the literature. (As we will explain below, H\"ormander \cite{HormanderRiesz}\cite[\S5]{Hormander68} has a general treatment of such asymptotics, but when applied to the case at hand they fail to be sharp.) 
Our main motivation is to examine how the large-$\Sigma$ asymptotics of the $N_{3,k;\Lambda,\bfA}$ are encoded in the Fourier transform of $N_{d;\Lambda,\bfA}$.
All other motivation aside, it seems worthwhile to give a relatively self-contained presentation.

Although we cannot deduce from the results below
anything interesting about the asymptotics of $N_{3;\Lambda,\bfA}$ beyond the easy result 
\begin{equation} 
	N_{3;\Lambda,\bfA} = (4\pi/3) \operatorname{Covol}(\Lambda)^{-1}\Sigma^3 + O(\Sigma^{3/2+\epsilon}), 
\end{equation} 
the asymptotics of $N_{3,1;\Lambda,\bfA},N_{3,2;\Lambda,\bfA},\cdots$ are suggestive and serve well to illustrate the oscillatory nature of the remainder. The results below are therefore complementary to estimates of the second moment
\begin{equation} 
\int_0^\Sigma |N_{3,0;\bbZ^3}(\sigma)-(4\pi/3)\sigma^3|^2 \dd \sigma 
\end{equation}
as $\Sigma\to\infty$, such as those found in \cite{lau}\cite{hulse}, which give no information regarding the sign of the discrepancy $N_{3,0;\bbZ^3}(\Sigma)-(4\pi/3)\Sigma^3$. 

A closely related notion to the iterated integrals above is that of ``spectral Riesz means.'' For $k\in \bbN$, the $k$th \emph{Riesz mean} of the counting function $N_{3;\Lambda,\bfA}$ is defined as 
\begin{equation}
	R^k_\Sigma N_{3;\Lambda,\bfA}(\Sigma) = k!\Sigma^{-k} N_{3,k;\Lambda,\bfA} (\Sigma) = \int_0^\Sigma \Big( 1 - \frac{\sigma}{\Sigma}\Big)^k \dd N_{3;\Lambda,\bfA}(\sigma) . 
\end{equation}
This is the $k$th Riesz mean of the eigenvalue counting function for the magnetic Schr\"odinger operator $\triangle + 2i \bfA\cdot \nabla + \lVert \bfA \rVert^2$ on the flat 3-torus $\bbT^3 = \bbR^3/\Lambda^*$, where $\triangle$ is the positive semidefinite Laplacian and $\bfA\cdot \nabla = A_x \partial_x+A_y\partial_y+A_z\partial_z$. 
It is a theorem of H\"ormander \cite{HormanderRiesz}\cite[\S5]{Hormander68} that the $k$th spectral Riesz mean for any closed Riemannian manifold has an asymptotic expansion in powers of $\Sigma$, with an $O(\Sigma^{d-k-1})$ error term. Thus, the $k$th iterated integral of the eigenvalue counting function has an asymptotic expansion in powers of $\Sigma$ with an $O(\Sigma^{d-1})$ error term. This error is of size $O(\Sigma^2)$ in the $d=3$ case of interest here. 
Instead, the analysis below results in an $O(\Sigma)$ error for $k\geq 2$.

In discussing Weyl remainders, it can be useful \cite{FullingEstrada} to identify the eigenvalue counting with a tempered distribution.
Specifically, letting $\calS(\bbR) \subset C^\infty(\bbR)$ denote the Fr\'echet space of Schwartz functions, we conflate the function $N_{d,k;\Lambda,\bfA}:\bbR\to \bbR$ with the mapping 
\begin{equation}
	\calS(\bbR)\ni \chi \mapsto \int_0^\infty \chi(\sigma) N_{d,k;\Lambda,\bfA}(\sigma) \dd \sigma \in \bbC. 
	\label{eq:misc_l3z}
\end{equation}
(It is straightforward to verify that \cref{eq:misc_l3z} is continuous and therefore defines a tempered distribution.)
One reason for doing so is that we will be working with the Fourier transform. Another is that  it trivializes matters of convergence when talking about infinite series; convergence holds (sometimes only) after ``smearing in $\Sigma$,'' i.e.\ in the weak topology generated by the functionals given by integration against the various Schwartz functions. 
Indeed, some of the formulas below are \textit{not} convergent in the ordinary sense, but they make sense when the left-hand and right-hand sides are interpreted distributionally.
This is not an obstacle to getting \emph{some} ordinarily convergent formula, as \Cref{rem:distributional} indicates, as we can just smear against a convenient choice of test function, but the smearing might obscure relevant features or otherwise overcomplicate the formulas. When convergence does hold in a stronger topology than that of $\calS'(\bbR)$, this can be proven as a secondary result. Thus, one advantage of this approach is that it separates the underlying harmonic analysis from secondary technical considerations regarding convergence.
 
We will work with two spaces of generalized functions:
\begin{itemize}
	\item $\dot{\calS}'(\bbR^{\geq 0}) \subset \calS'(\bbR)$ is the (linear and topological) subspace consisting of tempered distributions supported on $[0,\infty)=\smash{\bbR^{\geq 0}}$, that is the set of tempered distributions $u\in \calS'(\bbR)$ such that $u(\chi) = 0$ whenever the support of $\chi$ is disjoint from $[0,\infty)$. (In a much more interesting setting than the 1D setting considered here, Melrose \cite{Melrose} calls them ``supported distributions.'') We endow $\smash{\dot{\calS'}(\bbR^{\geq 0})}$ with the subspace topology induced by the inclusion \begin{equation} 
		\dot{\calS}'(\bbR^{\geq 0})\hookrightarrow \calS'(\bbR),
	\end{equation} 
	which makes $\dot{\calS}'(\bbR^{\geq 0})$ into an LCTVS. 
	\item $\calS'(\bbR^{\geq 0}) = \dot{\calS}(\bbR^{\geq 0})^*$ (LCTVS-dual), where  $\smash{\dot{\calS}(\bbR^{\geq 0}_\Sigma)}$ denotes the Fr\'echet space of  all smooth functions $\chi:[0,\infty)_\Sigma\to \bbC$ which vanish to infinite order at the origin and which decay rapidly as $\Sigma\to\infty$, along with all of their derivatives. 
	(Melrose calls the elements of $\calS'(\bbR^{\geq 0})$ ``extendable distributions.'')
\end{itemize}
Some facts about these spaces are reviewed in \S\ref{sec:app}. References for the theory of distributions include \cite{RS}\cite{Melrose}\cite{HAP}\cite{Dijk2013}.
For our purposes, it is not necessary to specify a collection of seminorms generating the topology of $\calS'(\bbR^{\geq 0})$ --- it suffices to note that $\calS'(\bbR^{\geq 0})$ is canonically identifiable with 
\begin{equation} 
	\dot{\calS}'(\bbR^{\geq 0}) / \bbC[\partial]\delta = \{ u \bmod \bbC[\partial]\delta : u \in \dot{\calS}'(\bbR^{\geq 0})\}, 
\end{equation} 
where $\bbC[\partial]\delta$ denotes the set of linear combinations of a Dirac $\delta$-function and its derivatives. Indeed, given any $u\in \smash{\dot{\calS}'(\bbR^{\geq 0})}$, so that $u$ is a map $u:\calS(\bbR)\to \bbC$, we can restrict $u$ to $\smash{\dot{\calS}(\bbR^{\geq 0})}$. This defines an onto linear map 
\begin{equation} 
	\dot{\calS}'(\bbR^{\geq 0})\to \calS'(\bbR^{\geq 0})
\end{equation} 
whose kernel consists precisely of $\bbC[\partial] \delta$. Thus, we have an induced invertible linear map 
\begin{equation} 
	\dot{\calS}'(\bbR^{\geq 0}) / \bbC[\partial]\delta \to \calS'(\bbR^{\geq 0}),
\end{equation} 
and it is a homeomorphism.
See \cite[Lemma 1.2]{Melrose} for the general case.

We can identify any element of $\smash{\cup_{K\in \bbR} \langle \Sigma \rangle^K L^1[0,\infty)_\Sigma}$
with an element of $\smash{\calS'(\bbR_\Sigma^{\geq 0})}$ in the same way as done with ordinary tempered distributions.

So, for each $d\in \bbN^+$, $k\in \bbN$, full rank $\Lambda\subset \bbR^d$, and $\bfA\in \bbR^d$, the functions $N_{d;\Lambda,\bfA}$ and $N_{3,k;\Lambda,\bfA}$ can be interpreted as elements of $\smash{\dot{\calS}'(\bbR^{\geq 0})}$ or $\calS'(\bbR^{\geq 0})$.
(A clash of notational conventions will lead to us being inconsistent with regards to factors of the Heaviside step function 
\begin{equation}
	\Theta:\bbR\to \{0,1\}, \quad \Theta(\sigma) = 1_{\sigma\geq 0}.
\end{equation} 
When considering $N_{3,k;\Lambda,\bfA}$ as a tempered distribution, we must write the step functions explicitly, but when we are working in $\smash{\dot{\calS}'(\bbR^{\geq 0})}$ we will typically omit them.) 

\begin{remark*}
	We follow the notational convention of Melrose in using overdots, e.g.\ `$\calS'(\bbR^{\geq 0})$' to denote the dual of $\dot{\calS}(\bbR^{\geq 0})$ and `$\dot{\calS}'(\bbR^{\geq 0})$' to denote the dual of $\calS(\bbR^{\geq 0})$. So, according to this convention, the `$'$' should not be read as dualization. 
	We will mostly drop the `$\bmod \,\bbC[\partial]\delta$' that should be written when specifying elements of $\dot{\calS}'(\bbR^{\geq 0}) / \bbC[\partial]\delta$. Hence, we will notationally conflate extendable distributions with supported distributions extending them.
	In addition, as discussed above, we will make ample use of conventional abuses of notation allowing the identification of locally integrable, polynomially growing functions with tempered distributions. 
	Regarding formal variables, when we write `$N_{3;\Lambda}(\Sigma)\in \calS'(\bbR_\Sigma)$' (or anything similar), we mean that $N_{3;\Lambda}$ is identifiable with an element of $\calS'(\bbR)$, with $\Sigma$ playing a formal role, notationally speaking. The (optional) subscripts just declare which variable is being integrated against where such disambiguation is felt to be useful.
\end{remark*}

In the latter half of our presentation, \S\ref{sec:as_part_1} and \S\ref{sec:as_part_2}, we turn to the toroidal half-wave trace. Because $N_{3;\Lambda,\bfA}(\Sigma)\notin L^1(\bbR_\Sigma)$, its Fourier transform must be defined in the sense of Schwartz;
explicitly, $\calF N_{3;\Lambda,\bfA}: \calS(\bbR)\to \bbC$ is  given by
\begin{equation}
	\calF N_{3;\Lambda,\bfA}( \chi) = \int_0^\infty N_{3;\Lambda,\bfA}(\Sigma) \calF \chi(\Sigma) \dd \Sigma
	\label{eq:FN3}
\end{equation} 	
for any $\chi \in \calS(\bbR)$, where our conventions in defining $\calF:\calS(\bbR)\to \calS(\bbR)$ are 
\begin{equation} 
	 \quad \calF \chi(\sigma) = \int_{-\infty}^{+\infty} e^{-i\tau \sigma} \chi(\tau) \dd \tau. 
\end{equation} 
(We will use $\tau$ to denote the Fourier dual variable to $\Sigma,\sigma$.)
\Cref{eq:FN3} is the Fourier transform with respect to $\Sigma$, as opposed to the coordinate $ \Sigma^2$, the latter of which would correspond to the heat kernel rather than the wave kernel. 
While $\calF N_{3;\Lambda,\bfA}$ is a well-defined tempered distribution, it is not a function. It has singularities, and we will keep track of them.
The Laplace transform of $N_{3;\Lambda,\bfA}(\Sigma)$ in $\Sigma^2$ (which is well-defined as an improper integral) is essentially a power of a Jacobi theta function, therefore an automorphic form after analytic continuation, and this entails that the Fourier transform of $N_{3;\Lambda,\bfA}(\Sigma)$ in $\Sigma^2$ (the distributional boundary value of a modular form at the real axis) has full singular support. It is a tempered distribution, but it is nowhere locally a function. 
In contrast, $\calF N_{3;\Lambda,\bfA}(\tau)$ has an isolated singularity at $\tau = 0$ and actually has a Laurent series there. It is a function except at a discrete set of times, around each of which it can be expanded in Laurent series:
\begin{theorem}
	\label{thm:main5}
	Setting $\mu_{3;\Lambda,\bfA}(\Sigma) = \sum_{{\bmSigma} \in \Lambda} \delta(\Sigma-|\bmSigma+\bfA|) \in \calS'(\bbR_\Sigma)$, the toroidal half-wave trace $\calF \mu_{3;\Lambda,\bfA}(\tau)\in \calS'(\bbR_\tau)$ is given by 
	\begin{equation}
		\calF \mu_{3;\Lambda,\bfA}(\tau) = \frac{8\pi i}{\operatorname{Covol}(\Lambda)}\Big[ \frac{\pi}{(\tau-i0)^3} + \tau \sum_{\bfk\in \Lambda^* \backslash \{\bf0\}}  \frac{ \cos(2\pi \bfA\cdot \bfk)}{((\tau-i0)^2-4\pi^2 |\bfk|^2 )^2} \Big],
		\label{eq:misc_000}
	\end{equation}
	where the sum is unconditionally convergent in $\calS'(\bbR_\tau)$. 
\end{theorem}
Integrating this using \Cref{prop:Fourier_integration}, we ge get an explicit formula 
\begin{equation}
	\calF N_{3;\Lambda,\bfA}(\tau) = \frac{8\pi^2 }{\operatorname{Covol}(\Lambda)} \frac{1}{(\tau-i0)^4} + \frac{8\pi  }{\operatorname{Covol}(\Lambda)} \sum_{\bfk\in \Lambda^* \backslash \{\bf0\}}  \frac{\cos( 2\pi\bfA\cdot \bfk)}{((\tau-i0)^2-4\pi^2 |\bfk|^2 )^2}
	\label{eq:misc_001}
\end{equation}
for $\calF N_{3;\Lambda,\bfA}$.

The half-wave trace on general compact Riemannian manifolds was studied by Chazarain \cite{Chazarain1974}, H\"ormander \cite{Hormander68}, and Duistermaat \& Guillemin \cite{DG}, and together they established the precise singularity structure of the half-wave trace of any closed Riemannian manifold. We will not use this general theory, since we can just compute $\calF N_{3;\Lambda,\bfA}(\tau)$ explicitly, but the singularity structure in \cref{eq:misc_000} is in accord with these more general results. 
That $\calF N_{3;\Lambda,\bfA}(\tau)$ has an isolated singularity (of known form) at the origin is useful, because it allows us to  separate the polynomial growth of $N_{3;\Lambda,\bfA}(\Sigma)$ as $\Sigma\to \infty$ (coming from the pole of the Fourier transform at the origin) and an oscillatory remainder, coming from the rest. From this perspective, the additional result needed to deduce \emph{sharp} asymptotics for the $N_{d,k;\Lambda,\bfA}$ (for sufficiently large $k$) is global quantitative control on the half-wave trace, with a polynomial rate of growth (where the meaning of ``sufficiently large'' is determined by the degree of the rate of growth).

Via the Poisson summation formula, we can compute out the oscillatory contribution to $N_{d,0;\Lambda,\bfA}$ exactly, for every $d\geq 1$: it is an infinite sum of Bessel functions, yielding a formula for $N_{d;\Lambda,\bfA}(\Sigma)$ generalizing (in a weak, distributional sense) the well-known formula for $N_{2;\bbZ^2}(\Sigma)$ --- cf.\ \cite[\S4.4]{IwaniecKowalski}. 
While we cannot estimate the size of this oscillatory term precisely (a problem familiar from the Gauss circle problem, which is still open, despite the analogous formula being known for a century), the situation changes upon integration in $\Sigma$, since upon passing to the Fourier transform this corresponds to a weighting by $1/\tau$, which improves decay (away from the origin) as $\tau \to \pm \infty$. Hence, after integrating, we can better estimate the oscillatory term -- elementarily too, at least in the $d=3$ case we restrict attention to. However, the division by $\tau$ makes more severe the singularity of the half-wave trace at the origin, adding new terms to the Laurent series there. Applying the inverse Fourier transform then yields new polynomial terms. The end result is:

\begin{figure}[t]
	\begin{center} 
		\includegraphics[width=\textwidth]{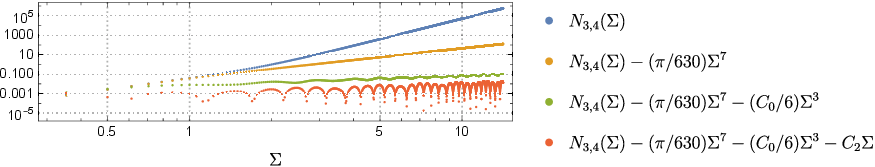}
	\end{center} 
	\caption{A \texttt{loglog} plot of $N_{3,4}(\Sigma) = N_{3,4;\bbZ^3,\bf0}(\Sigma)$ for $\Sigma \in \{\sqrt{\lambda/8}:\lambda\in \bbN,1\leq \lambda \leq 1600\}$, along with the results of subtracting from $N_{3,4}$ the non-oscillatory terms in \cref{eq:iterated_formula}. The last function plotted is just $o_4(\Sigma)=o_{4;\bbZ^3;\bf0}(\Sigma)$, as defined by \cref{eq:okmain}. The constants $C_0=C_{0;\bbZ^3,\bf0}$ and $C_2=C_{2;\bbZ^3,\bf0}$ were numerically approximated to within about $\smash{10^{-9}}$, and we ignore the resultant experimental error (of order approximately $10^{-6}$) introduced in computing the functions above.}
	\label{fig}
\end{figure}

\begin{theorem}
	\label{thm:main1}
	For each $k=0,1,2,3,\cdots$, 
	\begin{equation}
		N_{3,k;\Lambda,\bfA}(\Sigma) = \frac{8\pi}{(3+k)!} \frac{\Sigma^{3+k}}{\operatorname{Covol}(\Lambda)} + \sum_{m=0}^{k-1} \frac{C_{k-1-m;\Lambda,\bfA}}{m!} \Sigma^m + o_{k;\Lambda,\bfA}(\Sigma) 
		\label{eq:iterated_formula}
	\end{equation}
	as an element of $\dot{\calS}'(\bbR^{\geq 0})$, 	where $\{C_{j;\Lambda,\bfA}\}_{j=0}^\infty \subset \bbR$ is defined by (the absolutely convergent series)
	\begin{equation}
		C_{j;\Lambda,\bfA} = \frac{1}{\operatorname{Covol}(\Lambda)}
		\begin{cases}
			0 & (j \; \mathrm{ odd}), \\
			+(2j+4)(2\pi)^{-j-3} \sum_{\bfk\in \Lambda^*\backslash \{\bf0\}}\cos(2\pi \bfA\cdot \bfk)|\bfk|^{-4-j} & (j=0 \bmod 4), \\
			-(2j+4)(2\pi)^{-j-3} \sum_{\bfk\in \Lambda^* \backslash \{\bf0\}} \cos( 2\pi\bfA\cdot \bfk) |\bfk|^{-4-j} & (j=2 \bmod 4), 
		\end{cases}
		\label{eq:c35}
	\end{equation}
	and  $o_{k;\Lambda,\bfA} \in \dot{\calS}'(\bbR^{\geq 0})$ is given by 
	\begin{multline}
		o_{k;\Lambda,\bfA}(\Sigma) = - \frac{1}{\pi \operatorname{Covol}(\Lambda)}  \sum_{\bfk\in \Lambda^* \backslash \{\bf0\}} \frac{1}{|\bfk|^2} \frac{1}{(2\pi |\bfk|)^{k}} \cos(2\pi \bfA\cdot \bfk) \\
		\times \begin{cases}
			(-1)^{\frac{k}{2}}\;\;\,\big[\Sigma \cos(2\pi |\bfk| \Sigma) - \frac{k+1}{2\pi |\bfk| } \sin(2\pi |\bfk| \Sigma)\big] &(k\in 2\bbN)\\
			(-1)^{\lceil \frac{k}{2} \rceil}\big[\Sigma \sin(2\pi |\bfk| \Sigma) - \frac{k+1}{2\pi |\bfk|} \cos(2\pi |\bfk| \Sigma)\big] &(k\notin 2\bbN).
		\end{cases}
	\label{eq:okmain}
	\end{multline}
	The series in \cref{eq:okmain} is unconditionally summable in $\dot{\calS}'(\bbR^{\geq 0})$. 
	
	If $k\geq 2$, then \cref{eq:iterated_formula} holds in the ordinary sense, for each $\Sigma\geq 0$ (with the series in \cref{eq:okmain} ordinarily absolutely convergent). 
	Moreover, if $k\geq 2$, $o_{k;\Lambda,\bfA}:[0,\infty) \to \bbR$ is continuous. It vanishes at the origin, and, for any compact $K\subset \smash{\bbR^{\geq 0}}$, the series in \cref{eq:okmain}, restricted to $ K$, converges uniformly to it.
\end{theorem}
\begin{corollary}
	\label{thm:main2}
	As $\Sigma\to\infty$, 
	\begin{equation}
		N_{3,1;\Lambda,\bfA}(\Sigma) = \frac{\pi}{3} \frac{\Sigma^4}{\operatorname{Covol}(\Lambda)} + O_{\Lambda,\bfA}(\Sigma \log \Sigma) 
		\label{eq:n31as}
	\end{equation}
	holds. 
\end{corollary}
\begin{corollary}
	\label{thm:main3}
	For each $k\geq 2$, 
	\begin{equation}
		N_{3,k;\Lambda,\bfA}(\Sigma) = \frac{8\pi}{(3+k)!} \frac{\Sigma^{3+k}}{\operatorname{Covol}(\Lambda)} + \sum_{m=2}^{k-1} \frac{C_{k-1-m;\Lambda,\bfA}}{m!} \Sigma^m + O_{k,\Lambda,\bfA}(\Sigma) 
		\label{eq:iterated_asymptotics}
	\end{equation}
	for all $\Sigma\geq 0$. 
\end{corollary}
\noindent The subscripts on the big-O's denote that the bounds depend on the listed parameters in some unexamined way. Below, we will leave this dependence implicit. 

Finally, we check that there are no serious cancellations in \cref{eq:okmain}:
\begin{theorem}
	\label{thm:main4}
	For no integer  $k \geq 1$ does there exist an $\alpha=\alpha_{\bfA,\Lambda,k}\in \bbR$ and $\epsilon>0$ such that 
	\[
	N_{3,k;\Lambda,\bfA}(\Sigma) = \frac{8\pi}{(3+k)!} \frac{\Sigma^{3+k}}{\operatorname{Covol}(\Lambda)} +   \sum_{m=2}^{k-1} \frac{C_{k-1-m;\Lambda,\bfA}}{m!} \Sigma^m + \alpha \Sigma + O(\Sigma^{1-\epsilon})
	\]
	as $\Sigma\to\infty$. 
\end{theorem}

See \Cref{fig} for an illustration of \Cref{thm:main1}, \Cref{thm:main3}, \Cref{thm:main4} in the instructive case $k=4$, $\Lambda=\bbZ^3$, $\bfA=0$. The oscillatory character of $N_{3,4;\bbZ^3}(\Sigma) - (\pi/630) \Sigma^7 - (C_{0;\Lambda}/6)\Sigma^3-C_{2;\Lambda} \Sigma$ is clearly visible, as is the linear growth rate of the amplitudes in the terms of the oscillatory remainder.

The computation used below to prove the first of these results is very standard --- it is just a (somewhat nondirect) application of the Poisson summation formula, keeping track of the possible singularity at the origin. It can be understood in terms of the ``method of images'' for solving PDE on $\bbT^d$, but this just amounts to an application of the Poisson summation formula anyways. A brisk presentation of the core of the argument in the case $\Lambda=\bbZ^3$, $\bfA=0$ can be found in \cite[\S2]{Chamizo1995}\cite[\S4.4]{IwaniecKowalski}, and a presentation of van der Corput's version can be found in \cite[\S4.5.2]{Pinsky}. The case of a general lattice or nonzero shift does not require any new ideas. 
The results for $k>0$  require in our presentation a secondary argument in \S\ref{sec:exp} (though after some smoothing in $\Sigma$ this can be replaced by an appeal to the precise amount of regularity in $\bmSigma$ required in \cite[Lemma 2.1]{Chamizo1995}\cite[Corollary 4.8]{IwaniecKowalski}). 
The slightly technical point here is that in the traditional presentation of these sorts of computations, one wants to apply the Poisson summation formula to the function 
\begin{equation} 
	f_\Sigma(\bmSigma)=(\Sigma-|\bmSigma|)_+^k \in C^0(\bbR^d_{\bmSigma}),
\end{equation} 
which is not smooth and therefore not Schwartz. Once can appeal to \cite[Chp. VII, Cor. 2.6]{SteinWeiss}
in order to justify the application, but this requires first computing out $\calF f_\Sigma$. The computation of this Fourier transform is analogous to the computation in \S\ref{sec:exp}.

\Cref{thm:main2} is extracted from \Cref{thm:main1} in \S\ref{sec:as_part_1} via a very elementary smoothing argument. 

We elaborate briefly on how the results above give more precise information regarding spectral Riesz means on flat 3-tori than what is known for general manifolds.
H\"ormander's asymptotic expansion for the spectral Riesz means of $\bbT^3$ is contained in \Cref{thm:main2}, \Cref{thm:main3}. The asymptotic expansions in \Cref{thm:main2} and \Cref{thm:main3} are more precise, however --- besides being completely explicit about what the coefficients in H\"ormander's asymptotic expansion are, the error term is only of size $O(\Sigma)$ for $k\geq 2$ and $O(\Sigma \log \Sigma)$ for $k=1$, instead of $O(\Sigma^2)$. Moreover, for $k\geq 2$, \Cref{thm:main1} gives an absolutely convergent trigonometric sum for the error term, which is very particular to the torus. To boot, whereas it is typically not known when H\"ormander's $O(\Sigma^2)$ estimate is sharp, the estimate here is actually sharp, as \Cref{thm:main4} shows. We think that \Cref{thm:main3} is interesting in part as a sharp estimate of some spectral Riesz means, and to our knowledge this result is the first sharp estimate of higher Riesz means for some non-Zoll manifolds of dimension $>1$. Although we do not carry it out, an analogous computation to the one below gives sharp asymptotics for the ``off-diagonal'' spectral Riesz means of $\bbT^3$. See \cite{HormanderRiesz}\cite{Fulling} for more on spectral Riesz means in general.

Conceptually, the feature of $\bbT^3$  -- and flat tori more generally -- that allows proving this refined asymptotic expansion is the slow growth of the (say $2$-fold regularized) toroidal half-wave trace (as measured in an appropriate sense). Thus, while our proof is partly based on an explicit computation, a similar result will hold on any closed Riemannian manifold on which the half-wave trace grows polynomially with respect to an appropriate Sobolev-type norm measuring negative regularity. See \cite{sussman} for a related discussion of Weyl remainders themselves. 
The detailed analysis of the toroidal half-wave trace found in \S\ref{sec:as_part_2} is  difficult to locate in the literature (if it exists, it is unknown to the authors). We use the results in this section to prove \Cref{thm:main4}. We remark in closing that, even in the $\Lambda=\bbZ^3,\bfA=0$ case, the summation formulas \cite[Lemma 2.1]{Chamizo1995}\cite[Corollary 4.8]{IwaniecKowalski} do not suffice for the proof of \Cref{thm:main5} (they only determine $\calF\mu_{3;\Lambda,\bfA}$ modulo a polynomial), so it is not clear that one can avoid the additional technical details in the treatment below. 

\section{An explicit functional formula for $N_{d;\Lambda,\bfA}$}
\label{sec:core}

The Poisson summation formula for $\Lambda$ says that 
\begin{equation}
	\sum_{\bmSigma \in [\Lambda+\bfA]} f(\bmSigma) = \sum_{\bmSigma \in \Lambda}f_{\bfA}(\bmSigma) = \frac{1}{\operatorname{Covol}(\Lambda)}  \sum_{\bfk \in \Lambda^*}  \calF f_\bfA ( 2\pi \bfk) = \frac{1}{\operatorname{Covol}(\Lambda)}  \sum_{\bfk \in \Lambda^*}  e^{2\pi i \bfk \cdot \bfA} \calF f ( 2\pi \bfk)
	\label{eq:Poisson}
\end{equation}
for all $f\in \calS(\bbR^d)$, where $f_{\bfA}(\bmSigma)=f(\bmSigma+\bfA)$. 
In other words, both sums converge unconditionally -- trivially, in this case, because $f$ and $\calF f$ are Schwartz -- and their limits agree. 

Let $\bmSigma = (\sigma_1,\ldots,\sigma_d)$  and $r=(\sigma_1^2+\cdots+\sigma_d^2)^{1/2} = \lVert \bmSigma \rVert$. 

Suppose now that $f$ has the form $f(\sigma_1,\ldots,\sigma_d)=F(r)$  for some $F\in \calS(\bbR)$ whose odd-order derivatives all vanish at zero. For example, any 
\begin{equation} 
	F \in \dot{\calS} (\bbR^{\geq 0}) = \dot{C}^\infty(\bbR^{\geq 0})\cap \calS (\bbR)
	\label{eq:Fcon}
\end{equation} 
has this property, as does any Schwartz function that is constant in some neighborhood of the origin. 

For such $F$, $f$ is actually Schwartz, so the Poisson summation formula applies, yielding 
\begin{align}
	\sum_{\bmSigma \in [\Lambda+\bfA]} F(|\bmSigma|) &= \frac{1}{\operatorname{Covol}(\Lambda)} \sum_{\bfk \in \Lambda^*}  e^{2\pi i \bfk \cdot \bfA} \calF f ( 2\pi \bfk) \\ &= \frac{1}{\operatorname{Covol}(\Lambda)}  \sum_{\bfk \in \Lambda^*} e^{2\pi i \bfk \cdot \bfA} \int_{0}^\infty F(r) r^{d-1} \int_{\bbS^{d-1}} e^{-2\pi i r\bfk\cdot  \bmtheta}  \dd\!\operatorname{Area}_{\bbS^{d-1}} (\bmtheta) \dd r.  
	\label{eq:m53}
\end{align}
In the expression above is a standard form of the Bessel function of the first kind: 
\begin{equation}
	\int_{\bbS^{d-1}} e^{i \bfw \cdot \bmtheta } \dd\! \operatorname{Area}_{\bbS^{d-1}} (\bmtheta) = \begin{cases} 
		(2\pi)^{\nu + 1 } |\bfw|^{-\nu} J_\nu(|\bfw|)  & (\bfw \neq 0), \\
		\operatorname{Area}(\bbS^{d-1}) & (\bfw = 0), 
	\end{cases} 
	\label{eq:29g}
\end{equation}
where $\nu = d/2-1$.  See \cite[pg.\ 198]{GS}\cite[pg. 154]{SteinWeiss}.
Substituting this into \cref{eq:m53},  
\begin{multline}
	\sum_{\bmSigma \in [\Lambda+\bfA]} F(|\bmSigma|) =  \frac{A}{\operatorname{Covol}(\Lambda)} \int_0^\infty F(r) r^{d-1} \dd r \\ + \frac{1 }{\operatorname{Covol}(\Lambda)}  \sum_{\bfk \in \Lambda^* \backslash \{{\bf0}\}}  2\pi\cos(2\pi \bfk \cdot \bfA)\int_0^\infty F(r) r^{d/2} | \bfk|^{-\nu} J_\nu(2\pi r |\bfk|) \dd r,
	\label{eq:misc_l30}
\end{multline}
where $A=\operatorname{Area}(\bbS^{d-1})$
and we used the symmetry of $\Lambda^*$ to replace the $\exp(2\pi i \bfk\cdot \bfA)$ with $\cos(2\pi \bfk\cdot \bfA)$. 

Note that the formal series $\sum_{\bmSigma \in \Lambda} \delta(\Sigma -  |\bmSigma+\bfA|) \in \calS'(\bbR_\Sigma)^{\Lambda}$  is unconditionally summable to a tempered distribution (which is in fact a Borel measure), which we  denote $
\Delta(\Sigma)=\sum_{\bmSigma \in \Lambda} \delta(\Sigma - |\bmSigma+\bfA|)$. 
The left-hand side of \cref{eq:misc_l30} can therefore be written as follows:
\begin{equation}
	\sum_{\bmSigma \in [\Lambda+\bfA]} F(| \bmSigma|) = \int_{-\infty}^{+\infty} F(\Sigma) \sum_{\bmSigma \in \Lambda} \delta(\Sigma - |\bmSigma+\bfA|) \dd \Sigma = \Delta(F), 
\end{equation}
where the integral is formal. Similarly -- letting, for $\bfk \neq 0$, $\calJ_{\bfk} \in \calS'(\bbR)$ denote the tempered distribution $\chi \mapsto \int_{\bbR} \chi(\Sigma) \calJ_{\bfk}(\Sigma) \dd \Sigma$ given by integration against the function
\begin{equation}  
	\calJ_{\bfk}(\Sigma) = 
	\begin{cases} 
		2\pi \Sigma^{d/2} |\bfk|^{-\nu} J_\nu(2\pi |\bfk|\Sigma ) & (\Sigma>0) \\
		0 & (\Sigma \leq 0), 
	\end{cases} 
	\label{eq:Jdef}
\end{equation} 
(which we denote $\calJ_{\bfk}(\Sigma) = 2\pi  \Theta(\Sigma) 	\Sigma^{d/2} |\bfk|^{-\nu} J_\nu(2\pi \Sigma |\bfk|)$) -- the summand on the right-hand side of \cref{eq:misc_l30} is $\cos(2\pi i \bfk \cdot \bfA)\calJ_{\bfk}(F)$. 
(Note that the subscript `$\bfk$' in $\calJ_{\bfk}$ is not related to the subscript `$\nu$' in $J_\nu$.) 

The formal series $\smash{\sum_{\bfk \in \Lambda^*\backslash \{\bf0\}} \cos(2\pi  \bfk \cdot \bfA) \calJ_{\bfk} \in \calS'(\bbR)^{\Lambda\backslash \{\bf0\}}}$ is unconditionally summable in $\calS'(\bbR^{\geq 0})$, as an integration-by-parts argument shows. 
The identities above show that 
\begin{equation}
	\Delta(F)  = \frac{A}{\operatorname{Covol}(\Lambda)}  \int_0^\infty F(r) r^{d-1} \dd r + \frac{1}{\operatorname{Covol}(\Lambda)}  \Big[  \sum_{\bfk \in \Lambda^* \backslash  \{{\bf0}\}} \cos(2\pi \bfk \cdot \bfA) \calJ_{\bfk} \Big] (F),
	\label{eq:misc_7h5}
\end{equation}
for all $F$ as above. So, we have arrived at \cite[Theorem 4.6]{IwaniecKowalski}:
\begin{propositionp}
	\label{prop:dif}
	As elements of $\calS'(\bbR^{\geq 0}_\Sigma)$, $ \operatorname{Covol}(\Lambda) \Delta(\Sigma) = A \Sigma^{d-1}+ \sum_{\bfk \in \Lambda^*\backslash \{{\bf0}\}} \cos(2\pi  \bfk \cdot \bfA) \calJ_{\bfk}(\Sigma)$. 
\end{propositionp} 
We will need to pay attention to Dirac terms, but first we observe the following, which is suggested by formally integrating the identity in \Cref{prop:dif}:
\begin{proposition}
	\label{prop:if}
	As elements of $\calS'(\bbR^{\geq 0}_\Sigma)$, we have 
	\begin{align} N_{d;\Lambda,\bfA}(\Sigma) &= \Sigma^d \frac{\operatorname{Vol} \bbB^d}{\operatorname{Covol}(\Lambda)}  + \frac{\Sigma^{d/2}}{\operatorname{Covol}(\Lambda)} \sum_{\bfk \in \Lambda^*\backslash \{{\bf0}\}} \cos(2\pi  \bfk \cdot \bfA) |\bfk|^{-d/2} J_{d/2}(2\pi |\bfk| \Sigma),
	\label{eq:Nform} 
	\end{align}  
	where the sum is unconditionally convergent in $\calS'(\bbR^{\geq 0}_\Sigma)$.  
\end{proposition} 
\begin{proof}
	We only need to rigorously justify the formal integration. We first check that the left-hand and right-hand sides of \cref{eq:Nform}, \textit{considered as elements of $\calS'(\bbR^{\geq 0})$}, have the same derivative. We observe:
	\begin{enumerate}[label=(\Roman*)]
		\item $\partial N_{d;\Lambda,\bfA}(\Sigma) = \Delta(\Sigma)$ when both sides are considered as elements of $\calS'(\bbR)$, so this continues to hold when both sides are continued as elements of $\calS'(\bbR^{\geq 0})$, and 
		\item the sum  
		\begin{equation} 
			\Sigma^{d/2} \sum_{\bfk \in \Lambda^*\backslash \{{\bf0}\}}  \cos(2\pi  \bfk \cdot \bfA) |\bfk|^{-d/2} \Theta(\Sigma) J_{d/2}(2\pi |\bfk| \Sigma) \in \calS'(\bbR_\Sigma^{\geq 0})^{\Lambda^*\backslash \{\bf0\}}
			\label{eq:misc_kk}
		\end{equation} 
		is unconditionally convergent in $\calS'(\bbR^{\geq 0})$ by an integration-by-parts argument, and the derivative of the extendable distribution in \cref{eq:misc_kk} is $\smash{\sum_{\bfk \in \Lambda^*\backslash \{\bf0\}}  \cos(2\pi  \bfk \cdot \bfA) \calJ_{\bfk}}$. 
	\end{enumerate}
	So, indeed, the two sides of \cref{eq:Nform} have the same derivative in $\calS'(\bbR^{\geq 0})$. 
	This is equivalent to saying that, when interpreted as elements of $\smash{\dot{\calS}'(\bbR^{\geq 0})}$, they have the same derivative modulo $\bbC[\partial]\delta$.

	By \Cref{prop:mod_ker}, we conclude that, as an element of $\smash{\calS'(\bbR^{\geq 0}_\Sigma)}$, 
	\begin{equation}
		N_{d;\Lambda,\bfA}(\Sigma) = c+\Sigma^d \frac{\operatorname{Vol} \bbB^d}{\operatorname{Covol}(\Lambda)} + \frac{\Sigma^{d/2}}{\operatorname{Covol}(\Lambda)} \sum_{\bfk \in \Lambda^*\backslash \{{\bf0}\}} |\bfk|^{-d/2}  \cos(2\pi  \bfk \cdot \bfA) J_{d/2}(2\pi |\bfk| \Sigma)
		\label{eq:misc_nd4}
	\end{equation}
	for some constant $c = c_{\Lambda,\bfA}\in \bbC$, which we have to show is zero. (We are omitting the $\Theta(\Sigma)$ in \cref{eq:misc_nd4}.)
	
	For $\chi \in C_{\mathrm{c}}^\infty(\bbR)$ that is constant in a neighborhood of the origin, we calculate $N_{d;\Lambda,\bfA}(\partial \chi) = \int_0^\infty N_{d;\Lambda,\bfA}(\Sigma)\chi'(\Sigma) \dd \Sigma$ in two ways:
	\begin{itemize}
		\item 
		Since $\partial N_{d;\Lambda,\bfA}(\Sigma) = \Delta(\Sigma)$ as elements of $\calS(\bbR)$,  $N_{d;\Lambda,\bfA}(\partial \chi)=-\Delta(\chi)$.
		By \cref{eq:misc_l30}, 
		\begin{align}
			 -\Delta(\chi) &=- A \int_0^\infty \frac{\chi(\Sigma) \Sigma^{d-1} \dd \Sigma}{\operatorname{Covol}(\Lambda)} - 2\pi  \sum_{\bfk \in \Lambda^*\backslash \{\bf0\}}\frac{ \cos(2\pi  \bfk \cdot \bfA)}{|\bfk|^{\nu} \operatorname{Covol}(\Lambda)} \int_0^\infty \chi(\Sigma) \Sigma^{d/2} J_\nu(2\pi  |\bfk| \Sigma) \dd \Sigma \\
			&= \frac{\operatorname{Vol} \bbB^d}{\operatorname{Covol}(\Lambda)}  \int_0^\infty \chi'(\Sigma) \Sigma^{d} \dd \Sigma+ \sum_{\bfk \in \Lambda^*\backslash \{\bf0\}}^\infty \frac{ \cos(2\pi  \bfk \cdot \bfA)}{|\bfk|^{d/2} \operatorname{Covol}(\Lambda)} \int_0^\infty \chi'(\Sigma) \Sigma^{d/2} J_{d/2}(2\pi |\bfk| \Sigma) \dd \Sigma .
			\label{eq:misc_g11}
		\end{align}
		\item Applying \cref{eq:misc_nd4}, which we can do because $N_{d;\Lambda,\bfA}(\partial \chi) = N_{d;\Lambda,\bfA}(\Theta \partial \chi)$ and $\Theta \partial \chi \in \dot{\calS}(\bbR^{\geq 0})$, yields 
		\begin{multline}
			N_{d;\Lambda,\bfA}(\partial \chi) = - c \chi(0)+ \frac{\operatorname{Vol} \bbB^d}{\operatorname{Covol}(\Lambda)}  \int_0^\infty \chi'(\Sigma) \Sigma^{d} \dd \Sigma \\ + \sum_{\bfk \in \Lambda^*\backslash \{\bf0\}}^\infty \frac{ \cos(2\pi  \bfk \cdot \bfA)}{|\bfk|^{d/2} \operatorname{Covol}(\Lambda)} \int_0^\infty \chi'(\Sigma) \Sigma^{d/2} J_{d/2}(2\pi |\bfk| \Sigma) \dd \Sigma .
			\label{eq:misc_g42}
		\end{multline}
	\end{itemize}
	Comparing \cref{eq:misc_g11} and \cref{eq:misc_g42}, we deduce that $c\int_0^\infty \chi'(\Sigma) \dd \Sigma = -c \chi(0) =0$. Since we can choose $\chi$ as above that is nonzero at zero, we conclude that $c=0$. 
\end{proof}

\begin{remark}
	When $d=2$, $\sum_{n=1}^\infty r_2(n) n^{-1/2} J_1(2\pi \sqrt{n} \Sigma)$ converges (conditionally) for $\Sigma>0$ with $\Sigma^2 \notin \bbZ$ (as shown by Hardy).
	If $d\geq 3$ then there does not exist any $\Sigma>0$ such that the sum on the right-hand side of \cref{eq:Nform} converges in the ordinary sense, hence it must be interpreted ``distributionally:'' given any Schwartz function $\chi \in \calS(\bbR)$, 
	\begin{multline}
		\int_{0}^\infty \chi(\Sigma) N_{d;\Lambda}(\Sigma) \dd \Sigma = \frac{\operatorname{Vol} \bbB^{d}}{\operatorname{Covol}(\Lambda)}\int_0^\infty \chi(\Sigma) \Sigma^d \dd \Sigma  \\ + \frac{1}{\operatorname{Covol}(\Lambda)}\sum_{\bfk \in \Lambda^*\backslash \{\bf0\}} \frac{1}{|\bfk|^{d/2}} \int_0^\infty \chi(\Sigma) \Sigma^{d/2} J_{d/2}(2\pi |\bfk| \Sigma) \dd \Sigma, 
	\end{multline}
	where the sum on the right-hand side is absolutely convergent. For example, since $N_{d;\bbZ^d}$ is constant on intervals of the form $[\sqrt{n}, \sqrt{n+1})$ for $n\in \bbN$, choosing $\chi \in C_{\mathrm{c}}^\infty(\bbR)$ supported in $[\sqrt{n}, \sqrt{n+1})$ with $\int_\bbR \chi(\Sigma) \dd \Sigma = 1$, 
	\begin{equation}
		N_{d;\bbZ^d}(\sqrt{n}) = \operatorname{Vol} \bbB^{d}\int_0^\infty \chi(\Sigma) \Sigma^d \dd \Sigma  + \sum_{n=1}^\infty \frac{r_d(n)}{n^{d/4}} \int_0^\infty \chi(\Sigma) \Sigma^{d/2} J_{d/2}(2\pi \sqrt{n} \Sigma) \dd \Sigma,   
	\end{equation}
	where $r_d(n) = \{\bfk \in \bbZ^d : |\bfk| = n\}$. 
	\label{rem:distributional}
\end{remark}

Whenever $d \in 2\bbN+1$, the Bessel function $J_{d/2}(z)$ can be written as a linear combination of trigonometric functions whose coefficients are Laurent polynomials in $z$. For $d=3$, we have 
\begin{equation} 
	J_{3/2}(z) = \sqrt{\frac{2}{\pi z}}\Big(-\cos(z) + \frac{1}{z} \sin(z)\Big)
\end{equation} 
for all $z\in \bbC\backslash (-\infty,0]$. 
Substituting this in to \cref{eq:Nform}, we get: 
\begin{propositionp}
	\label{prop:k0case}
	As an element of $\calS'(\bbR^{\geq 0}_\Sigma)$, $N_{3;\Lambda,\bfA}(\Sigma)$ is given by
	\begin{equation}
		 \frac{4\pi}{3} \frac{\Sigma^3}{\operatorname{Covol}(\Lambda)}  - \frac{1}{\pi \operatorname{Covol}(\Lambda)} \sum_{\bfk \in \Lambda^*\backslash \{\bf0\}} \frac{1}{|\bfk|^2} \cos(2\pi \bfk\cdot \bfA)\Big[ \Sigma\cos(2\pi |\bfk| \Sigma) - \frac{1}{2\pi |\bfk|} \sin(2\pi |\bfk| \Sigma) \Big],
	\end{equation}
	where the sum on the right-hand side is (unconditionally) convergent in $\calS'(\bbR^{\geq 0}_\Sigma)$.
\end{propositionp}

\section{An explicit formula for $N_{3,k;\Lambda,\bfA}$, $k\geq 1$}
\label{sec:exp}

We now proceed to iterate the formal integration leading us from \Cref{prop:dif} to \Cref{prop:if}. We have not, however, yet proven that the series on the right-hand side of \cref{eq:Nform} converges unconditionally in $\dot{\calS}'(\bbR^{\geq 0})$, i.e.\ in $\calS'(\bbR)$ (after inserting $\Theta$'s) (we only noted that -- by integration-by-parts -- they converge unconditionally in $\calS'(\bbR^{\geq 0}) \cong \dot{\calS}'(\bbR^{\geq 0})/\bbC[\partial]\delta$). And, the computations in the previous section, leading up to \cref{eq:misc_l30}, were only done for Schwartz functions $F \in \calS(\bbR)$ that are the sum of an even function and an element of $\dot{\calS}(\bbR^{\geq 0})$.  For proving \Cref{prop:if}, this sufficed, but for repeated integrations we need to allow arbitrary Schwartz $F$. 

For each $k\in \bbN$, let $\partial^{-k} : \dot{\calS}'(\bbR^{\geq 0}) \to \dot{\calS}'(\bbR^{\geq 0})$ denote the well-defined set-theoretic inverse of $\partial^k$, i.e.\ the $k$th power of the operator $\partial^{-1}$ defined in \cref{eq:int}.  Heuristically, $\partial^{-k} \chi$ is given by $k$-fold iterated integrations of $\chi$ from $0$ and this holds literally when applied to integrable functions, and consequently 
\begin{equation} 
	\partial^{-k} N_{3;\Lambda,\bfA}(\Sigma) = N_{3,k;\Lambda,\bfA}(\Sigma),
	\label{eq:misc_zeh}
\end{equation} 
where both sides are interpreted as elements of $\smash{\dot{\calS}'(\bbR^{\geq 0})}$.

As a preliminary step in computing $N_{3,k;\Lambda,\bfA}$ via \cref{eq:misc_zeh}, we define $\{\tilde{N}_{3,k;\Lambda,\bfA} \}_{k\in \bbN}\subset  \calS'(\bbR^{\geq 0})$ such that 
\begin{itemize}
	\item $\tilde{N}_{3,0;\Lambda,\bfA} = N_{3,0;\Lambda,\bfA}$ \textit{in $\calS'(\bbR^{\geq 0})$}, 
	\item and 
	\begin{equation} 
		\tilde{N}_{3,0;\Lambda,\bfA}(\Sigma) =\partial^{k} \tilde{N}_{3,k;\Lambda,\bfA}(\Sigma),
		\label{eq:misc_zey}
	\end{equation} 
\end{itemize}
also in $\calS'(\bbR^{\geq 0})$. We will check that the formal series defining $\tilde{N}_{3,k;\Lambda,\bfA}$ converges unconditionally in $\smash{\dot{\calS}'(\bbR^{\geq 0})}$ for $k\geq 2$, and by differentiating we will conclude the same for $k=0,1$, hence essentially proving that the series on the right-hand side of \Cref{prop:dif} converges absolutely in some negative regularity Sobolev space. 

This will allow us to interpret \cref{eq:misc_zey} as an identity in $\dot{\calS}'(\bbR^{\geq 0})$. 

Given now that $\tilde{N}_{3,0;\Lambda,\bfA}$ is a well-defined element of $\dot{\calS}'(\bbR^{\geq 0})$, the fact that $\tilde{N}_{3,0;\Lambda,\bfA} = N_{3,0;\Lambda,\bfA}$ in $\calS'(\bbR^{\geq 0})$  means that 
\begin{equation}
	\tilde{N}_{3,0;\Lambda,\bfA}(\Sigma) = P(\partial) \delta(\Sigma) + N_{3,0;\Lambda,\bfA}(\Sigma) 
\end{equation}
in $\dot{\calS}(\bbR^{\geq 0}_\Sigma)$ for some polynomial $P=P_{\Lambda,\bfA}$. 
\Cref{eq:misc_zey} then determines $\tilde{N}_{3,k;\Lambda,\bfA}$ in terms of $N_{3,0;\Lambda,\bfA}$ modulo polynomials and derivatives of Dirac $\delta$-functions (the results of integrating $\bbC[\partial]\delta$). 

It will turn out that we defined $\tilde{N}_{3,k;\Lambda,\bfA}$ so that, as elements of $\dot{\calS}'(\bbR^{\geq 0})$, 
\begin{equation} 
	\tilde{N}_{3,k;\Lambda,\bfA} = N_{3,k;\Lambda,\bfA}, 
\end{equation}	
meaning that $P=0$. This will be proven by checking that 
$\lim_{\Sigma\to 0^+} \tilde{N}_{3,k;\Lambda,\bfA}(\Sigma) = 0$
for all $k\geq 1$. 

This second statement will be clear for all $k\geq 2$, since $\smash{\tilde{N}_{3,k;\Lambda,\bfA}}$ will be a continuous function on all of $\bbR$. The $k=1$ case is a little more delicate, but the potential logarithmic singularity is not severe enough to cause trouble. 

\subsection{Formal Integration}
Set 
\begin{multline}
	\tilde{N}_{3,k;\Lambda,\bfA}(\Sigma)  = \frac{8\pi}{(3+k)!} \frac{\Sigma^{3+k}}{\operatorname{Covol}(\Lambda)} \Theta(\Sigma) \\ -\frac{1}{\pi \operatorname{Covol}(\Lambda)}   \sum_{\bfk\in \Lambda^*\backslash\{\bf0\}} \frac{1}{|\bfk|^2} \cos(2\pi\bfk\cdot \bfA) \partial^{-k} \Big[ \Sigma \cos(2\pi |\bfk| \Sigma) \Theta(\Sigma) - \frac{1}{2\pi |\bfk|} \sin(2\pi |\bfk| \Sigma) \Theta(\Sigma) \Big].
	\label{eq:Ntilde_def}
\end{multline}
The series is unconditionally summable in $\smash{\calS'(\bbR^{\geq 0})}$, so $\tilde{N}_{3,k;\Lambda,\bfA}$ is a well-defined element of $\calS'(\bbR^{\geq 0})$.

All of the computations in this section will be done in $\calS'(\bbR^{\geq 0})$. 

The  function 
\begin{equation} 
	\partial^{-k} [ \Sigma \cos(2\pi |\bfk| \Sigma) \Theta(\Sigma) - (2\pi |\bfk|)^{-1} \sin(2\pi |\bfk| \Sigma) \Theta(\Sigma)]
\end{equation} 
is a $k$-fold iterated definite integral of $\sigma \cos(2\pi |\bfk| \sigma) - (2\pi |\bfk|)^{-1} \sin(2\pi |\bfk| \sigma)$ along $\sigma \in [0,\Sigma]$.

Let $\Sigma_{\bfk} = 2\pi |\bfk| \Sigma$, so that $\partial_\Sigma = 2\pi |\bfk| \partial_{\Sigma_{\bfk}}$. Then, $\partial_\Sigma^{-k} = (2\pi)^{-k} |\bfk|^{-k} \partial_{\Sigma_{\bfk}}^{-k}$, so 
\begin{align}
	\partial^{-k} \Big[ \Sigma \cos( \Sigma_{\bfk}) - \frac{1}{2\pi |\bfk|} \sin(\Sigma_{\bfk})\Big] = (2\pi)^{-(k+1)} |\bfk|^{-(k+1)}  \partial^{-k}_{\Sigma_{\bfk}} [  \Sigma_{\bfk} \cos(\Sigma_{\bfk}) -  \sin(\Sigma_{\bfk})]. 
	\label{eq:3lg}
\end{align}
It can be seen inductively that there exist $\alpha_k,\beta_k,\gamma_k,\delta_k \in \bbR$ and polynomials $Q_k\in \bbR[\Sigma]$ such that 
\begin{equation}
	\partial^{-k}_{\Sigma} [  \Sigma \cos(\Sigma) -  \sin(\Sigma)] = Q_k(\Sigma) + \alpha_k \Sigma \cos \Sigma + \beta_k \Sigma \sin \Sigma + \gamma_k \cos \Sigma + \delta_k \sin \Sigma.
\end{equation}
Plugging this into the computation above, 
\begin{multline}
	\tilde{N}_{3,k;\Lambda,\bfA}(\Sigma) = \frac{8\pi}{(3+k)!} \frac{\Sigma^{3+k}}{\operatorname{Covol}(\Lambda)}  - \frac{1}{\pi \operatorname{Covol}(\Lambda)} \sum_{\bfk\in \Lambda^*\backslash \{\bf0\}} \frac{1}{|\bfk|^2} \frac{1}{(2\pi |\bfk|)^{k+1}} \cos(2\pi \bfk\cdot \bfA)\\ \times [ Q_k(\Sigma_{\bfk}) + \alpha_k \Sigma_{\bfk} \cos \Sigma_{\bfk}  + \beta_k \Sigma_{\bfk} \sin \Sigma_{\bfk} + \gamma_k \cos \Sigma_{\bfk} + \delta_k \sin \Sigma_{\bfk}] ,
	\label{eq:k45}
\end{multline}
where we are now omitting the factors of $\Theta(\Sigma)$, both sides being interpreted as elements of $\calS'(\bbR^{\geq 0})$. 

So, in order to compute $\tilde{N}_{3,k;\Lambda,\bfA}$ for $k\geq 1$, all that needs to be done is compute $\alpha_k,\beta_k,\gamma_k,\delta_k$, and $Q_k$.

Set $\alpha_0 = 1, \beta_0 = 0, \gamma_0 = 0, \delta_0 = -1$.  Since $\alpha_k$, $\beta_k$, $\gamma_k$, $\delta_k$ can be defined in terms of $\alpha_{k-1},\beta_{k-1},\gamma_{k-1},\delta_{k-1}$ via a system of linear equations whose coefficients do not depend on $k$, 
\begin{equation}
	\begin{pmatrix}
		\alpha_k \\ \beta_k \\ \gamma_k \\ \delta_k 
	\end{pmatrix} = 
	\begin{pmatrix}
		m_{11} & m_{12} & m_{13} & m_{14} \\
		m_{21} & m_{22} & m_{23} & m_{24} \\
		m_{31} & m_{32} & m_{33} & m_{34} \\
		m_{41} & m_{42} & m_{43} & m_{44} \\
	\end{pmatrix}^k \begin{pmatrix}
		\alpha_0 \\ \beta_0 \\ \gamma_0 \\ \delta_0
	\end{pmatrix}
	\label{eq:ind}
\end{equation}
for some matrix $M=\{m_{ij}\}_{i,j=1}^4$ with entries $m_{ij} \in \bbR$. From
\begin{equation}
	\int_0^\Sigma \sigma \cos(\sigma) \dd \sigma = -1 + \cos(\Sigma) + \Sigma \sin (\Sigma), \quad 
	\int_0^\Sigma \sigma \sin(\sigma) \dd \sigma = - \Sigma \cos(\Sigma) + \sin(\Sigma),
\end{equation}
along with $\int_0^\Sigma \cos(\sigma) \dd \sigma =  \sin(\Sigma)$ and $\int_0^\Sigma \sin(\sigma) \dd \sigma =  1 - \cos(\Sigma)$, we conclude that the only nonzero matrix entries are $m_{12} = -1$, $m_{21} = +1$, $m_{31} = +1$, $m_{34} = -1$, $m_{42}=+1$, and $m_{43}=+1$.
That is, 
\begin{equation}
	M = 
	\begin{pmatrix}
		0 & -1 & 0 & 0 \\
		+1 & 0 & 0 & 0 \\
		+1 & 0 & 0 & -1 \\
		0 & +1 & +1 & 0 \\
	\end{pmatrix} = 
	\begin{pmatrix}
		J & 0 \\
		1 & J
	\end{pmatrix},
\end{equation}
where $J \in \bbR^{2\times 2}$ is the matrix corresponding to a clockwise $90^\circ$ rotation of $\bbR^2$. 
We inductively  deduce that, for all $k\in \bbN$,  
\begin{equation} 
	M^k = \begin{pmatrix}
		J^{k} & 0 \\
		kJ^{k-1} & J^{k}
	\end{pmatrix}.
\end{equation} 
We have $J^k = 1$ if $k=0 \bmod 4$, $J^k = J$ if $k=1 \bmod 4$, $J^k = -1$ if $k=2 \bmod 4$, and $J^k = -J$ if $k=3 \bmod 4$; so, $M^k$ is given by 
\begin{equation}
	\begin{pmatrix}
		0 & -1 & 0 & 0 \\
		+1 & 0 & 0 & 0 \\
		+k & 0 & 0 & -1 \\
		0 & +k & +1 & 0 \\
	\end{pmatrix}, \; 
	\begin{pmatrix}
		-1 & 0 & 0 & 0 \\
		0 & -1 & 0 & 0 \\
		0 & -k & -1 & 0 \\
		+k & 0 & 0 & -1 \\
	\end{pmatrix}, \; \begin{pmatrix}
	0 & +1 & 0 & 0 \\
	-1 & 0 & 0 & 0 \\
	-k & 0 & 0 & +1 \\
	0 & -k & -1 & 0 \\
\end{pmatrix}, \;
\begin{pmatrix}
+1 & 0 & 0 & 0 \\
0 & +1 & 0 & 0 \\
0 & +k & +1 & 0 \\
-k & 0 & 0 & +1 \\
\end{pmatrix}
\end{equation}
in the four cases $k=1,2,3,4\bmod 4$ respectively. 

Using these in \cref{eq:ind}, we have $\alpha_k = \Re [i^k]$, $\beta_k = -\Im[i^k]$, $\gamma_k =(k+1) \Im [i^k]$, $\delta_k = - (k+1) \Re [i^k]$. 
From the computations above, we also see that 
\begin{equation}
	Q_k(\Sigma) = \int_0^\Sigma Q_{k-1}(\sigma)\dd \sigma - \alpha_{k-1} + \delta_{k-1} =  \int_0^\Sigma Q_{k-1}(\sigma)\dd \sigma -  (k+1)\Im[i^{k}].
\end{equation}
The $j$th component of this, $[Q_k]_j$, is given by $(j!)^{-1} \partial^j Q_k(0) = (j!)^{-1}Q_{k-j}(0) = - (k-j+1) \Im[i^{k-j}] / j!$ if $k-j \geq 1$ and $0$ otherwise. So, 
\begin{equation} 
	Q_k(\Sigma) = \sum_{j=0}^{k-1} (k-j+1) \Im[i^{k-j+2}]\frac{\Sigma^j}{j!}.
\end{equation} 
We therefore conclude that, for all $k\in \bbN$, 
\begin{multline}
	\tilde{N}_{3,k;\Lambda,\bfA}(\Sigma) = \frac{8\pi}{(3+k)!} \frac{\Sigma^{3+k}}{\operatorname{Covol}(\Lambda)}   -  \frac{1}{\pi \operatorname{Covol}(\Lambda)} \sum_{\bfk \in \Lambda^*\backslash \{\bf0\}} \frac{1}{|\bfk|^2} \frac{1}{(2\pi |\bfk| )^{k+1}} \cos(2\pi \bfk\cdot \bfA)\\ \Big[     2\pi|\bfk| \Re [i^k] \Sigma \cos (2\pi|\bfk|\Sigma)  - 2\pi|\bfk| \Im[i^k]\Sigma \sin (2\pi |\bfk| \Sigma)  + (k+1) \Im [i^k] \cos  (2\pi |\bfk| \Sigma) \\ - (k+1) \Re [i^k] \sin  (2\pi |\bfk| \Sigma)  
	 +\sum_{j=0}^{k-1} \frac{k-j+1}{j!} (2\pi|\bfk| \Sigma )^j \Im[i^{k-j+2}]  \Big] .
	\label{eq:misc_kj1}
\end{multline}
We may pull out the polynomial part of the sum and write (adding back in the factors of $\Theta$ for later reference) 
\begin{multline}
	\tilde{N}_{3,k;\Lambda,\bfA}(\Sigma) = \frac{8\pi}{(3+k)!} \frac{\Sigma^{3+k}}{\operatorname{Covol}(\Lambda)} \Theta(\Sigma) \\ + \frac{\Theta(\Sigma)}{\pi \operatorname{Covol}(\Lambda)} \sum_{j=0}^{k-1} \frac{k-j+1}{j!}  \frac{\cos(2\pi \bfk\cdot \bfA)}{(2\pi)^{k-j+1}} \Im[i^{k-j}] \Big( \sum_{\bfk \in \Lambda^*\backslash \{\bf0\}} \frac{1}{|\bfk|^{k-j+3}} \Big) \Sigma^j  + o_{k;\Lambda,\bfA}(\Sigma) , 
	\label{eq:misc_43h}
\end{multline} 
for 
\begin{multline}
	o_{k;\Lambda,\bfA}(\Sigma) = 
	- \frac{\Theta(\Sigma)}{\pi \operatorname{Covol}(\Lambda)} \sum_{\bfk\in \Lambda^* \backslash \{\bf0\} } \frac{1}{|\bfk|^2} \frac{\cos(2\pi \bfk\cdot \bfA)}{(2\pi |\bfk| )^{k}} \Big[ \Re [i^k] \Sigma \cos (2\pi|\bfk|\Sigma)  -  \Im[i^k]\Sigma \sin (2\pi |\bfk| \Sigma) \\ + \frac{k+1}{2\pi |\bfk|} \Im [i^k] \cos  (2\pi |\bfk| \Sigma) - \frac{k+1}{2\pi |\bfk|} \Re [i^k] \sin  (2\pi |\bfk| \Sigma) \Big].
	\label{eq:ok}
\end{multline} 

By comparing $o_{k;\Lambda,\bfA}$ with a volume integral:
\begin{lemmap}
	\label{lem:key}
	Let $k\geq 2$. 
	Given any compact $K\subseteq [0,\infty)$, the series on the right-hand side of \cref{eq:misc_kj1} is uniformly convergent in $K$, so $\smash{\tilde{N}_{3,k;\Lambda,\bfA}}$, as defined by \cref{eq:misc_43h}, can be considered as an element of $C^0[0,\infty)$. 
\end{lemmap}
Furthermore,  $\tilde{N}_{3,k;\Lambda,\bfA}(0) = 0$ for all $k\geq 2$, so $\tilde{N}_{3,k;\Lambda,\bfA}$ can be considered as an element of $C^0(\bbR)$.
Similar statements apply to $o_{k;\Lambda,\bfA}$.  

\subsection{Absence of Dirac Terms}

\begin{proposition}
	For each $k\geq 0$, the formal series 
	\begin{multline} 
		o_{k;\Lambda,\bfA}(\Sigma) = 
		- \frac{1}{\pi \operatorname{Covol}(\Lambda)} \sum_{\bfk\in \Lambda^*\backslash \{\bf0\}} \frac{1}{|\bfk|^2} \frac{\Theta(\Sigma)}{(2\pi |\bfk|)^{k}} \cos(2\pi \bfk\cdot \bfA)  \Big[ \Re [i^k] \Sigma \cos (2\pi |\bfk|\Sigma)  \\ -  \Im[i^k]\Sigma \sin (2\pi |\bfk| \Sigma) + \frac{k+1}{2\pi |\bfk|} \Im [i^k] \cos  (2\pi |\bfk| \Sigma) - \frac{k+1}{2\pi |\bfk| } \Re [i^k] \sin  (2\pi |\bfk| \Sigma) \Big]
		\label{eq:ok_main}
	\end{multline} 
is unconditionally summable in $\calS'(\bbR)$, hence \cref{eq:misc_43h} defines an element of $\calS'(\bbR)$. 
\end{proposition}
\begin{proof}
	Given \Cref{lem:key}, the proposition holds for $k\geq 2$. We now use this to deduce the $k=0,1$ cases. For $k\in \bbN^+$,  
	\begin{equation}
		\partial o_{k;\Lambda,\bfA}(\Sigma)  = o_{k-1;\Lambda,\bfA}(\Sigma) - \frac{1}{\pi \operatorname{Covol(\Lambda)}} \sum_{\bfk\in \Lambda^*\backslash \{\bf0\}} \frac{1}{|\bfk|^2} \frac{k+1}{(2\pi |\bfk|)^{k+1}} \cos(2\pi \bfk\cdot \bfA) \Im [i^k] \delta(\Sigma)
		\label{eq:misc_34h}
	\end{equation}
	at the level of formal series in $\dot{\calS}'(\bbR^{\geq 0})$. 
	
	For $k=2$, the second formal series in \cref{eq:misc_34h} is just zero, from which it follows that the series in \cref{eq:ok_main} is unconditionally summable in $\calS'(\bbR)$ for $k=1$ and that the derivative of \cref{eq:ok_main}, considered as an element of $\calS'(\bbR)$, is equal to the sum. 
	
	For $k=1$, the second formal series is absolutely convergent, from which it follows that the series in \cref{eq:ok_main} is unconditionally summable in $\calS'(\bbR)$ for $k=0$ and that the derivative of \cref{eq:ok_main} as an element of $\calS'(\bbR)$ is equal to the sum.
\end{proof}

The preceding proposition shows that \cref{eq:ok_main} defines, for each $k\in \bbN$, a tempered distribution, and moreover that \cref{eq:misc_34h} holds in the usual distributional sense.  
Manifestly,
\begin{equation}
	o_{k;\Lambda,\bfA}(\Sigma), \tilde{N}_{3,k;\Lambda,\bfA}(\Sigma) \in \dot{\calS}'(\bbR^{\geq 0})
\end{equation}
In other words, the given equation is an unconditionally convergent series in $\dot{\calS}'(\bbR^{\geq 0})$.

Reversing the formal integration that led us to \cref{eq:misc_43h},
\begin{equation} 
	\partial \tilde{N}_{3,k;\Lambda,\bfA} = \tilde{N}_{3,k-1;\Lambda,\bfA} 
\end{equation} 
in $\calS'(\bbR)$ for $k\in \bbN^+$. Let $\tilde{\Delta} = \partial \tilde{N}_{3,0;\Lambda,\bfA} \in \calS'(\bbR)$.  
This is given by 
\begin{equation}
	\tilde{\Delta}(\Sigma) = \frac{A \Sigma^{d-1}}{\operatorname{Covol}(\Lambda)} \Theta(\Sigma) + \frac{1}{\operatorname{Covol}(\Lambda)} \sum_{\bfk \in \Lambda^*\backslash \{{\bf0}\}} \cos(2\pi  \bfk \cdot \bfA) \calJ_{\bfk}(\Sigma),
\end{equation}
where the sum is unconditionally summable in $\calS'(\bbR)$. 
\begin{proposition}
	\label{prop:Delta_full}
	As an element of $\calS'(\bbR)$, 
	\begin{equation}
		\Delta(\Sigma) = \frac{4\pi \Sigma^2}{\operatorname{Covol}(\Lambda)} \Theta(\Sigma) + P(\partial)\delta(\Sigma) +  \frac{1}{\operatorname{Covol}(\Lambda)} \sum_{\bfk \in \Lambda^*\backslash \{\bf0\}} \cos(2\pi \bfk\cdot \bfA) \calJ_{\bfk}(\Sigma) 
		\label{eq:misc_klz}
	\end{equation}
	for some polynomial $P=P_{\Lambda,\bfA}$ whose even order terms are all zero. 
\end{proposition} 
\begin{proof}
	Consider the support of $E(\Sigma) = \Delta(\Sigma) - \tilde{\Delta}(\Sigma)$. 
	\begin{itemize}
		\item By \Cref{prop:dif}, $E$
		is supported on $(-\infty,0]$. 
		\item Since both $\Delta$ and $\tilde{\Delta}$ are supported on $[0,\infty)$, the same is true for $E$.
	\end{itemize}
	So, $E$ is supported at the origin.
	Since the only distributions supported on points are linear combinations of $\delta$-functions and derivatives thereof \cite[Theorem 5.5]{Dijk2013}, \cref{eq:misc_klz} holds for some polynomial $P$. 
	
	Given Schwartz $F\in \calS(\bbR)$, 
	\begin{multline}
		 \Delta(F) = \frac{4\pi}{\operatorname{Covol}(\Lambda)} \int_0^\infty \Sigma^2 F(\Sigma) \dd \Sigma + \sum_{j=0}^\infty [P]_j (-1)^j F^{(j)}(0) 
		\\ + \frac{1}{\operatorname{Covol}(\Lambda)}  \sum_{\bfk \in \Lambda^* \backslash \{\bf0\}} \int_0^\infty F(\Sigma) \cos(2\pi \bfk\cdot \bfA)\calJ_{\bfk}(\Sigma)  \dd \Sigma . 
	\end{multline}
	Applying this for arbitrary even $F$ and comparing with \cref{eq:misc_l30}, which held for any even $F \in \calS(\bbR)$, we conclude that the even order terms of $P$ all vanish. 
\end{proof}

Put differently, $\Delta(\Sigma) = P(\partial)\delta(\Sigma) + \tilde{\Delta}(\Sigma)$. Since integrals in $\dot{\calS}'(\bbR^{\geq 0})$ are unique, 
\begin{equation}
	N_{3,k;\Lambda,\bfA}(\Sigma) =
	\sum_{j=0}^k [P]_j  \frac{\Sigma^{k-j}}{(k-j)!} \Theta(\Sigma) + \sum_{j=k+1}^\infty [P]_j \delta^{(j-k-1)}(\Sigma) + \tilde{N}_{3,k;\Lambda,\bfA}(\Sigma)
	\label{eq:misc_zzz}
\end{equation} 
for each $k\geq 1$. By \Cref{lem:key}, $\tilde{N}_{3,k;\Lambda,\bfA}(\Sigma)$ is continuous for $k\geq 2$. Obviously, the same holds for $N_{3,k;\Lambda,\bfA}$ for $k\geq 2$, so \cref{eq:misc_zzz} forces that $[P]_2,[P]_3,[P]_4,\cdots$ are all zero (though we already knew that $[P]_2,[P]_4,\cdots$ were all zero). By the previous proposition, $[P]_0$ as well, so 
\begin{equation} 
	P(\partial) = C \partial
	\label{eq:misc_cpq}
\end{equation} 
for some $C\in \bbC$. 

\begin{proposition}
	The coefficient $C$ in \cref{eq:misc_cpq} is equal to $0$. 
\end{proposition}
\begin{proof}
	We examine \cref{eq:misc_zzz} for $k=1$. Taking into account the vanishing of $[P]_j$ for $j\neq 1$, this says that $\smash{N_{3,1;\Lambda,\bfA}(\Sigma) = C \Theta(\Sigma) + \tilde{N}_{3,1;\Lambda,\bfA}}$.  
	Since $C \Theta : \bbR\to \bbC$ is a piecewise continuous function and $N_{3,1;\Lambda,\bfA}:\bbR\to \bbR$ is a continuous function, the difference $\smash{\tilde{N}_{3,1;\Lambda,\bfA} = N_{3,1;\Lambda,\bfA} - C \Theta}$ is the distribution corresponding to a piecewise continuous function as well. 
	Then,  $\tilde{N}_{3,1;\Lambda,\bfA}(\Sigma) = N_{3,1;\Lambda,\bfA}(\Sigma) - C \Theta(\Sigma)$ holds in the ordinary sense for all $\Sigma\neq 0$, and we can write
	\begin{equation} 
		\lim_{\Sigma\to 0^+} \tilde{N}_{3,1;\Lambda,\bfA}(\Sigma) = \lim_{\Sigma\to 0^+} N_{3,1;\Lambda,\bfA}(\Sigma) - C = -C. 
	\end{equation} 
	Given any $\chi \in C_{\mathrm{c}}^\infty(\bbR)$, from the unconditional convergence, in $\calS'(\bbR)$, of the series defining $\tilde{N}_{3,1;\Lambda,\bfA}$, 
	\begin{multline}
		\operatorname{Covol}(\Lambda)\int_0^\infty  \chi(\Sigma) \tilde{N}_{3,1;\Lambda,\bfA}(\Sigma) \dd \Sigma = \frac{\pi}{3}  \int_0^\infty \chi(\Sigma) \Sigma^4 \dd \Sigma \\ + \frac{1}{2\pi^3} \Big[ \sum_{\bfk\in \Lambda^*\backslash \{\bf0\}} \frac{\cos(2\pi \bfk\cdot \bfA)}{|\bfk|^4} \Big] \int_0^\infty \chi(\Sigma) \dd \Sigma \\ 
		+ \frac{1}{2\pi^2} \sum_{\bfk\in \Lambda^*\backslash \{\bf0\}} \frac{1}{|\bfk|^{3}} \cos(2\pi \bfk\cdot \bfA)  \int_0^\infty \chi(\Sigma)  \Big[ \Sigma \sin(2\pi |\bfk| \Sigma) - \frac{1}{\pi |\bfk|} \cos(2\pi |\bfk| \Sigma)  \Big] \dd \Sigma. 		
		\label{eq:misc_hhh}
	\end{multline}
	Now suppose that $\chi$ is supported in $(0,+\varepsilon)$, $0<\varepsilon<1$.  
	In the following computations, we use big-$O$ notation with bounds independent of $\chi,\varepsilon$, and parameters $r,R$ (introduced below). We have the following bounds:
	\begin{itemize}
		\item The first is $\int_0^\infty \chi(\Sigma) \Sigma^4 \dd \Sigma = O(\varepsilon^4 \lVert \chi \rVert_{L^1})$. 
		\item For any $r>1$, 
		\begin{equation}
			\sum_{\bfk\in \Lambda^*\backslash \{{\bf0}\}, |\bfk| \leq r} \frac{1}{|\bfk|^{3}}\cos(2\pi \bfk\cdot \bfA) \int_0^\infty \chi(\Sigma) \Sigma \sin(2\pi |\bfk| \Sigma) \dd \Sigma   = O(\varepsilon \lVert \chi \rVert_{L^1} \log (1+r)).
		\end{equation}
		\item For any $r>1$, 
		\begin{multline}
			\sum_{\bfk \in \Lambda^*, |\bfk|>r} \frac{1}{|\bfk|^{3}}  \cos(2\pi \bfk \cdot \bfA)\int_0^\infty \chi(\Sigma) \Sigma \sin(2\pi |\bfk| \Sigma) \dd \Sigma \\ = \sum_{|\bfk|>r} \frac{1}{2\pi |\bfk|^4} \cos(2\pi \bfk \cdot \bfA) \int_0^\infty [\chi(\Sigma) + \Sigma \chi'(\Sigma)]  \cos(2\pi |\bfk| \Sigma) \dd \Sigma 
			\\ = O(r^{-1} (\lVert \chi \rVert_{L^1} + \varepsilon \lVert \chi' \rVert_{L^1})).
		\end{multline}
		\item Since $|1-\cos(2\pi |\bfk|\Sigma)| \leq 2\pi^2 |\bfk|^2 \Sigma^2$ for all $\Sigma \in \bbR$, 
		\begin{align*}
			\sum_{\bfk \in \Lambda^*\backslash \{{\bf0}\}, |\bfk|\leq R} \frac{1}{|\bfk|^4} \cos(2\pi \bfk\cdot \bfA) \int_0^\infty \chi(\Sigma) [1-\cos(2\pi |\bfk| \Sigma)] \dd \Sigma = O(\varepsilon^2 R \lVert \chi \rVert_{L^1})
		\end{align*}
		for any $R>1$. 
		\item On the other hand, just using a sup bound,
		\begin{align*}
			\sum_{\bfk\in \Lambda^*, |\bfk|>R} \frac{1}{|\bfk|^4} \cos(2\pi \bfk\cdot \bfA) \int_0^\infty \chi(\Sigma) [1-\cos(2\pi |\bfk| \Sigma)] \dd \Sigma = O(R^{-1} \lVert \chi \rVert_{L^1}).
		\end{align*}
	\end{itemize}
	Combining these estimates, we conclude that 
	\begin{equation}
		\int_0^\infty \chi(\Sigma) \tilde{N}_{3,1;\Lambda,\bfA}(\Sigma) \dd \Sigma = O((\varepsilon \log (1+r) +\varepsilon^2 R+ R^{-1} + r^{-1})\lVert \chi \rVert_{L^1} + r^{-1} \varepsilon \lVert \chi' \rVert_{L^1}).
		\label{eq:misc_49h}
	\end{equation}
	We now fix $\chi_1 \in C_{\mathrm{c}}^\infty(\bbR)$, supported in the open unit interval and satisfying $\int_0^1 \chi_1(\sigma)\dd \sigma = 1$, and we define rescalings $\smash{\chi_\varepsilon = \varepsilon^{-1} \chi_1(\varepsilon^{-1} \sigma)}$ for each $\varepsilon \in (0,1)$. Then $\lVert \chi_\varepsilon \rVert_{L^1} = \lVert \chi_1 \rVert_{L^1}$ and $\lVert \chi_\varepsilon ' \rVert_{L^1} = \varepsilon^{-1} \lVert \chi_1' \rVert_{L^1}$. So, \cref{eq:misc_49h} says 
	\begin{equation}
		\int_0^\infty \chi_{\varepsilon}(\Sigma) \tilde{N}_{3,1;\Lambda,\bfA}(\Sigma) \dd \Sigma = O(\varepsilon \log (1+r) + \varepsilon^2 R + R^{-1} + r^{-1}).
		\label{eq:misc_l45}
	\end{equation}
	We can choose  $r=R =  \varepsilon^{-1}$ to conclude that 
	\begin{equation} 
		\lim_{\varepsilon \to 0^+} \int_0^\infty \chi_{\varepsilon}(\Sigma) \tilde{N}_{3,1;\Lambda,\bfA}(\Sigma) \dd \Sigma = 0.
		\label{eq:misc_34c}
	\end{equation} 
	By the right-continuity of $\tilde{N}_{3,1;\Lambda,\bfA}$ at the origin, the left-hand side of \cref{eq:misc_34c} is $\lim_{\varepsilon\to 0^+} \tilde{N}_{3,1;\Lambda,\bfA}(\varepsilon)=-C$. We can therefore conclude that $C=0$. 
\end{proof}
Plugging this information into \cref{eq:misc_zzz}, we can deduce the main proposition of this section:
\begin{propositionp}
	\label{prop:N3k_final}
	For each $k \in \bbN$,  
	\begin{multline}
		N_{3,k;\Lambda,\bfA}(\Sigma) = \frac{8\pi}{(3+k)!} \frac{\Sigma^{3+k}}{\operatorname{Covol}(\Lambda)} \\ +\frac{1}{\pi \operatorname{Covol}(\Lambda)} \sum_{j=0}^{k-1} \frac{k-j+1}{j!}  \frac{1}{(2\pi)^{k-j+1}} \Im[i^{k-j}] \Big( \sum_{\bfk\in \Lambda^*\backslash \{\bf0\}} \frac{\cos(2\pi \bfk\cdot \bfA)}{|\bfk|^{k-j+3}} \Big) \Sigma^j + o_{k;\Lambda,\bfA}(\Sigma)
	\end{multline}
	as an element of $\dot{\calS}'(\bbR^{\geq 0})$, 
	where $o_{k;\Lambda,\bfA} \in \dot{\calS}'(\bbR^{\geq 0})$ is given by \cref{eq:ok}. 
\end{propositionp}
Comparing \Cref{prop:N3k_final} with \cref{eq:iterated_formula} and \cref{eq:c35}, we conclude \Cref{thm:main1}.

\section{The asymptotics of $N_{3,k;\Lambda,\bfA}$, $k\geq 1$, directly (sans sharpness)}
\label{sec:as_part_1}

For $k\geq 2$, there exist constants $c_k$ such that $|o_{k;\Lambda,\bfA}(\Sigma)| \leq c_k \Sigma$ for all $\Sigma \geq 0$. This follows from the convergence of 
\begin{equation}
	\sum_{\bfk \in \Lambda^* \backslash \{\bf0\}} \frac{1}{|\bfk|^{2+k}}
	\label{eq:3g}
\end{equation}
for $k\geq 2$, which can be proven e.g.\ via comparison with a volume integral.  We can therefore conclude that, for $k\geq 2$,  
\begin{multline}
	N_{3,k;\Lambda,\bfA}(\Sigma) = \frac{8\pi}{(3+k)!} \frac{\Sigma^{3+k}}{\operatorname{Covol}(\Lambda)} \\ + \frac{1}{\pi \operatorname{Covol}(\Lambda)} \sum_{j=0}^{k-1} \frac{k-j+1}{j!}  \frac{1}{(2\pi)^{k-j+1}} \Im[i^{k-j}] \Big( \sum_{\bfk\in \Lambda^*\backslash \{\bf0\}} \frac{\cos(2\pi\bfk\cdot \bfA)}{|\bfk|^{(k-j+3)}} \Big) \Sigma^j + O(\Sigma) , 
	\label{eq:h41}
\end{multline} 
as for all $\Sigma> 0$. Hence, \Cref{thm:main3} follows immediately from \Cref{thm:main1}. The remainder of this section is devoted to the $k=1$ case. 

For $k=1$, \cref{eq:3g} is harmonically divergent, and so if it is the case that the series in \cref{eq:ok} is convergent for given $\Sigma>0$, it is going to only be conditionally convergent. Nevertheless, we can still prove that 
\begin{equation}
	N_{3,1;\Lambda,\bfA}(\Sigma) = \frac{\pi\Sigma^4}{3\operatorname{Covol}(\Lambda)} + O(\Sigma \log \Sigma) 
	\label{eq:h42}
\end{equation}
as $\Sigma\to\infty$.
By \cref{eq:misc_43h}, 
\begin{equation}
	N_{3,1;\Lambda,\bfA}(\Sigma) = \frac{\pi\Sigma^{4}}{3\operatorname{Covol}(\Lambda)} + \frac{\Sigma}{2\pi^2 \operatorname{Covol}(\Lambda)} \sum_{\bfk \in \Lambda^*\backslash \{\bf0\}} \frac{\cos(2\pi\bfk\cdot \bfA)}{|\bfk|^{3}} \sin(2\pi |\bfk| \Sigma)  + O(\Sigma), 
	\label{eq:k55}
\end{equation}
in the sense that the left-hand side, viewed as an element of $\calS'(\bbR^{\geq 0})$, and 
\begin{equation} 
	\frac{\pi \Sigma^4}{3\operatorname{Covol}(\Lambda)} + \frac{\Sigma}{2\pi^2 \operatorname{Covol}(\Lambda)} \sum_{\bfk \in \Lambda^*\backslash \{\bf0\}} \frac{ \cos(2\pi \bfk\cdot \bfA)}{|\bfk|^3} \sin(2\pi |\bfk| \Sigma), 
\end{equation} 
also viewed as an element of $\smash{\calS'(\bbR^{\geq 0}_\Sigma)}$, differ by an element of $\smash{\calS'(\bbR^{\geq 0})}$ lying in $\langle \Sigma \rangle L^\infty(\bbR^{\geq 0})$. In fact, we can replace $\langle \Sigma \rangle L^\infty(\bbR^{\geq 0})$ with $\langle \Sigma \rangle C_{\mathrm{b}}^0(\bbR^{\geq 0})$ in the previous sentence, where  $C_{\mathrm{b}}^0\,(\bbR^{\geq 0})$ is the set of continuous \textit{bounded} functions on the closed half-line.

Let $\chi \in C_{\mathrm{c}}^\infty(\bbR)$ be a smooth, compactly supported, and real-valued function supported within the open unit interval $(0,1)$, satisfying
\begin{equation} 
	\int_0^1 \chi(\sigma) \dd \sigma = 1 \quad \text{ and } \quad \int_0^1 \chi(\sigma) \sigma^k \dd \sigma = 0
\end{equation} 
for all $k=1,2,3,4$. 
Such a function exists, and indeed $\chi$ can be constructed as a linear combination of $5$ bump functions via a linear algebraic argument:
\begin{itemize} 
	\item Let $\phi_1,\ldots,\phi_5 \in C_{\mathrm{c}}^\infty((0,1))$ and $\frakM$ be the $5\times 5$ matrix whose $(i,j)$th entry is $\smash{\int_0^1 \phi_i(\sigma) \sigma^j \dd \sigma}$. Then, $\chi$ can be taken as a linear combination of $\phi_1,\ldots,\phi_5$ if $\det \frakM \neq 0$. Fix $\phi_0 \in C_{\mathrm{c}}^\infty(\bbR)$ with $\phi_0(0)=1$ and total mass one, so that, setting 
	\[\phi_i(\sigma) = w^{-1}\phi_0(w^{-1}(\sigma-a_i))
	\] 
	for $i=1,2,3,4,5$, 
	$\frakM$ can be made arbitrarily close in operator norm to the Vandermonde matrix $\frakV$ on $a_1,\ldots,a_5$ (and thus $\det \frakM$ made arbitrarily close to $\det \frakV$) by taking $w\to 0^+$. If $a_1,\ldots,a_5$ are distinct, then $\det \frakV \neq 0$. So, for sufficiently small $w$, $\det \frakM \neq 0$ too. 
\end{itemize}

For each $\Sigma_0>1$, let $\chi_{\Sigma_0^2}(\sigma) = \varsigma_{\Sigma_0^2} \chi( \varsigma_{\Sigma_0^2} (\sigma-\Sigma_0))$ for $\varsigma_{\Sigma_0^2} = 1/((\Sigma_0^2+1)^{1/2}-\Sigma_0) = \langle \Sigma_0 \rangle(1+O(1/\langle \Sigma_0^2\rangle ))$. 
Note that  
\begin{equation}
	\int_{-\infty}^{+\infty} \chi_{\Sigma_0^2}(\sigma) \dd \sigma = \int_{\Sigma_0}^{\sqrt{\Sigma_0^2+1}} \chi_{\Sigma_0^2}(\sigma) \dd \sigma = 1.
\end{equation}
Then, 
\begin{equation}
	\int_{0}^{\infty} \chi_{\Sigma_0^2}(\Sigma) \int_0^\Sigma N_{3;\Lambda,\bfA}(\sigma) \dd \sigma \dd \Sigma  = \int_0^{\Sigma_0} N_{3;\Lambda,\bfA}(\sigma) \dd \sigma + \int_0^\infty \chi_{\Sigma_0^2}(\Sigma) \int_{\Sigma_0}^\Sigma N_{3;\Lambda,\bfA}(\sigma) \dd \sigma \dd \Sigma ,
\end{equation}
which is 
\begin{multline} 
	\int_0^{\Sigma_0} N_{3;\Lambda,\bfA}(\sigma) \dd \sigma +  N_{3;\Lambda,\bfA}(\Sigma_0)\int_0^\infty \chi_{\Sigma_0^2}(\Sigma) (\Sigma - \Sigma_0) \dd \Sigma \\ +\int_0^\infty \chi_{\Sigma_0^2}(\Sigma) \int_{\Sigma_0}^\Sigma (N_{3;\Lambda,\bfA}(\sigma)-N_{3;\Lambda,\bfA}(\Sigma_0)) \dd \sigma \dd \Sigma. 
\end{multline} 	
Per Gauss, the last term is $O(\Sigma_0)$. So, 
\begin{align}
	\begin{split} 
		\int_{0}^{\infty} \chi_{\Sigma_0^2}(\Sigma) \int_0^\Sigma N_{3;\Lambda,\bfA}(\sigma) \dd \sigma \dd \Sigma  
		&= \int_0^{\Sigma_0} N_{3;\Lambda,\bfA}(\sigma) \dd \sigma  \\ 
		&\qquad\qquad\qquad +  N_{3;\Lambda,\bfA}(\Sigma_0)\int_0^\infty \chi_{\Sigma_0^2}(\Sigma) (\Sigma - \Sigma_0) \dd \Sigma + O(\Sigma_0) \\
		&= \int_0^{\Sigma_0} N_{3;\Lambda,\bfA}(\sigma) \dd \sigma +  \frac{N_{3;\Lambda,\bfA}(\Sigma_0)}{\varsigma_{\Sigma_0^2}}  \int_0^1 \chi(\sigma )\sigma \dd \sigma +O(\Sigma_0) \\ 
		&= \int_0^{\Sigma_0} N_{3;\Lambda,\bfA}(\sigma) \dd \sigma + O(\Sigma_0). 
	\end{split} 
\end{align}
On the other hand, 
\begin{equation}
	\int_0^\infty \chi_{\Sigma_0^2}(\Sigma) \Sigma^4 \dd \Sigma = \int_{0}^{1} \chi(\sigma) \Big( \frac{\sigma}{\varsigma_{\Sigma_0^2}} - \Sigma_0 \Big)^4 \dd \sigma = \Sigma_0^4. 
\end{equation}
Integrating both sides of \cref{eq:k55} against $\chi_{\Sigma_0^2}$ therefore yields 
\begin{multline}
	\int_0^{\Sigma_0} N_{3;\Lambda,\bfA}(\sigma) \dd \sigma = \frac{\pi \Sigma_0^4}{3 \operatorname{Covol}(\Lambda)}  \\ + \frac{1}{\operatorname{Covol}(\Lambda)} \sum_{\bfk\in \Lambda^*\backslash \{\bf0\}}\frac{1}{|\bfk|^{3} } \cos(2\pi \bfk\cdot \bfA) \int_0^{\infty} \frac{\Sigma}{2\pi^2} \sin(2\pi |\bfk| \Sigma) \chi_{\Sigma_0^2}(\Sigma) \dd \Sigma  + O(\Sigma_0) ,
\end{multline} 
where the infinite series is absolutely summable. 
We split the sum into two halves, one with  $|\bfk|\leq R$ for to-be-determined $R>0$ and the (convergent) remainder. The first half is bounded as follows: noting that $\lVert \chi_{\Sigma_0} \rVert_{L^1} =\lVert \chi \rVert_{L^1}$,
\begin{multline}
	\Big|\sum_{\bfk \in \Lambda^*\backslash \{{\bf0}\} , |\bfk| \leq R} \frac{1}{|\bfk|^{3}} \cos(2\pi \bfk\cdot \bfA) \int_0^{\infty} \frac{\Sigma}{2\pi^2} \sin(2\pi |\bfk| \Sigma) \chi_{\Sigma_0^2}(\Sigma) \dd \Sigma \Big| \\ \leq \sum_{\bfk \in \Lambda^*\backslash \{{\bf0}\} , |\bfk| \leq R} \frac{1}{|\bfk|^{3}} \Big|\int_0^{\infty} \frac{\Sigma}{2\pi^2} \sin(2\pi |\bfk| \Sigma) \chi_{\Sigma_0^2}(\Sigma) \dd \Sigma \Big|,
\end{multline}
which is bounded above by 
\begin{equation} 
	\lVert \chi \rVert_{L^1} (\Sigma_0^2+1)^{1/2}  \sum_{\bfk \in \Lambda^* \backslash \{\bf0\},|\bfk| \leq R} \frac{1}{2\pi^2 |\bfk|^{3}}   \leq \Sigma_0 \log (R+1) \digamma
\end{equation} 
for some $\digamma=\digamma(\chi,\Lambda)>0$ that does not depend on $\Sigma_0,R$.
The second half is bounded as follows: 
\begin{multline}
	\Big|\sum_{\bfk \in \Lambda^* , |\bfk|>R} \frac{1}{|\bfk|^{3}} \cos(2\pi \bfk\cdot \bfA) \int_0^{\infty} \frac{\Sigma}{2\pi^2} \sin(2\pi |\bfk| \Sigma) \chi_{\Sigma_0^2}(\Sigma) \dd \Sigma \Big|  \\ \leq \Big|\sum_{\bfk \in \Lambda^*,|\bfk|>R} \frac{1}{|\bfk|^{4}} \cos(2\pi \bfk\cdot \bfA)  \int_0^{\infty} \frac{\Sigma}{4\pi^3} \cos(2\pi |\bfk| \Sigma) \chi_{\Sigma_0^2}'(\Sigma) \dd \Sigma \Big|  \\ +\Big|\sum_{\bfk \in \Lambda^*,|\bfk|>R}^\infty \frac{1}{|\bfk|^{4}} \cos(2\pi\bfk\cdot \bfA) \int_0^{\infty} \frac{1}{4\pi^3} \cos(2\pi |\bfk| \Sigma) \chi_{\Sigma_0^2}(\Sigma) \dd \Sigma \Big|. 
	\label{eq:k39}
\end{multline}
Using $\chi_{\Sigma_0^2}'(\Sigma) = \varsigma_{\Sigma_0^2}^2 \chi'(\varsigma_{\Sigma_0^2}(\Sigma-\Sigma_0))$, 
\begin{multline}
	\Big|\sum_{\bfk \in \Lambda^*,|\bfk|>R}^\infty \frac{1}{|\bfk|^{4}} \cos(2\pi \bfk\cdot \bfA) \int_0^{\infty} \frac{\Sigma}{4\pi^3} \cos(2\pi |\bfk| \Sigma) \chi_{\Sigma_0^2}'(\Sigma) \dd \Sigma \Big|
	\\ = \varsigma_{\Sigma_0^2}^2 \Big|\sum_{\bfk \in \Lambda^*,|\bfk|>R}^\infty \frac{1}{|\bfk|^{4}} \cos(2\pi \bfk\cdot \bfA) \int_0^{\infty} \frac{\Sigma}{4\pi^3} \cos(2\pi |\bfk| \Sigma) \chi'(\varsigma_{\Sigma_0^2}(\Sigma-\Sigma_0 )) \dd \Sigma \Big|
\end{multline}
is bounded above by 
\begin{equation}
	\digamma \Sigma_0^2 \sum_{\bfk\in \Lambda^*, |\bfk|>R}^\infty \frac{1}{|\bfk|^{4}}= O(\Sigma_0^2 R^{-1})
\end{equation}
for some $\digamma>0$ (not necessarily the same as before), 
and likewise for the second term on the right-hand side of \cref{eq:k39}. 
So, if $R>1$, 
\begin{equation}
	\int_0^{\Sigma_0} N_{3;\Lambda,\bfA}(\sigma) \dd \sigma = \frac{\pi \Sigma_0^4}{3 \operatorname{Covol}(\Lambda)} + O(\Sigma_0 \log R ) + O(\Sigma_0^2 R^{-1}). 
\end{equation}
We can therefore take $R = \Sigma_0$ and conclude that
\begin{equation} 
	\int_0^\Sigma N_{3;\Lambda,\bfA}(\sigma) \dd \sigma = \frac{\pi}{3} \frac{\Sigma^4}{\operatorname{Covol}(\Lambda)} + O(\Sigma \log \Sigma)
	\label{eq:95}
\end{equation} 
as $\Sigma\to\infty$.
We have therefore proven \Cref{thm:main2}.

\section{The asymptotics of $N_{3,k;\Lambda,\bfA}$, $k\geq 2$, via the Fourier transform}
\label{sec:as_part_2}

Missing from the previous section is a proof that the asymptotic expansions in \cref{eq:h41} and \cref{eq:h42} are the best asymptotic expansions possible (assuming, at least, that we are only allowed to use polynomials). In this section we will give another proof of these expansions, one which makes it easy to deduce sharpness. While somewhat more involved than the direct arguments above, it uses only structural features of $\calF N_{3,0;\Lambda,\bfA}$ which are guaranteed to hold on conceptual grounds -- the relation to the wave equation on the standard 3-torus -- so this approach has some appeal. 
Note that while we compute out $\calF N_{3,0;\Lambda,\bfA}$ fairly explicitly, we only use this computation in order to prove the estimate \Cref{prop:Fourier_bound} quickly. For the sake of \Cref{thm:main2}, \Cref{thm:main3}, \Cref{thm:main4}, what matters is the truth of this estimate, not the method of proof. 
We proceed via an $L^1,L^2$-based analysis of the Fourier transform $\calF N_{3;\Lambda,\bfA}=\calF N_{3;\Lambda,\bfA}(\tau)$, whose singularity structure we analyze in detail. The results below also serve to verify the hypotheses of the Tauberian theorem in \cite{sussman}, which, as stated there without proof, applies to the 3-torus. 

It is straightforward to compute that:
\begin{propositionp}
	\label{prop:HWT_comp}
	For each $r\geq 0$, the Fourier transform $\calF J\in \calS'(\bbR)$ of the function $J(\Sigma)=\Theta(\Sigma) \Sigma^{3/2} J_{1/2}(2\pi r \Sigma)$ is 
	\begin{equation}
		\calF J(\tau) = 4 i r^{1/2}\frac{ \tau}{( (\tau-i0)^2-4 \pi^2 r^2)^2 }= 4 i r^{1/2} \frac{ \tau-i0}{( (\tau-i0)^2-4 \pi^2 r^2)^2 },
	\end{equation}
	i.e.\ the limit $\lim_{\epsilon\to 0^+} H_\epsilon \in \calS'(\bbR)$ (in the topology of $\calS'(\bbR)$) of $H_\epsilon(\tau) = 4 i r^{1/2} (\tau-i \epsilon) ((\tau - i \epsilon)^2-4\pi^2  r^2)^{-2}$. 
\end{propositionp}

Let $j\in \bbR^+$. (For the present paper, it suffices to consider $j\in \bbN^+$, and for applications to Weyl's law more generally, $j\in \bbN^+/2$ should suffice.) For each $s\in \bbR$, we have a well-defined Fourier multiplier $\langle D \rangle^{-s} : \calS'(\bbR)\to \calS'(\bbR)$. For each $p\in [1,\infty)$, we let $\calS'(\bbR)\cap L^p_{\mathrm{loc}}(\bbR)$ denote the set of tempered distributions locally in $L^p(\bbR)$.
\begin{proposition}
	\label{prop:regcalc}
	Let $j\geq 1$. 
	For each $T,s \in \bbR$ and $p\in [1,\infty)$, $(\tau -T \pm i0)^{-j} \in \langle D \rangle^{-s}( \calS'(\bbR)\cap L^p_{\mathrm{loc}}(\bbR))$ if and only if $s< p^{-1} - j$.
\end{proposition}
\begin{proof}
	By the translation invariance of $\langle D \rangle^{-s}$, it suffices to consider the case $T=0$, that is to show that 
	\begin{equation}
		\langle D \rangle^s \frac{1}{(\tau \pm i0)^j} \in L^p_{\mathrm{loc}}(\bbR)
	\end{equation}
	under and only under the stated conditions. 
	Via complex conjugation, it suffices to consider the case of $P_{s,j}(\tau) = \langle D \rangle^s (\tau -i0)^{-j}$. 
	Claim: if $j+s$ is not an integer, then there exists (I) some $s,j$-dependent sequence $\{c_\ell\}_{\ell=0}^\infty$ containing only finitely many nonzero entries and (II) some $E \in C^0(\bbR)$ such that 
	\begin{equation}
		P_{s,j}(\tau)  = \sum_{\ell =0}^\infty   \frac{c_\ell}{(\tau - i0)^{j+s - \ell}} + E(\tau),
		\label{eq:l46}
	\end{equation}
	where $c_0 \neq 0$. 
	Indeed, $\calF^{-1} P_{s,j}(\sigma) = \langle \sigma \rangle^s \calF^{-1} P_{0,j}(\sigma) = \alpha_j\langle \sigma \rangle^s \Theta(\sigma) \sigma^{j-1}$ for some $\alpha_j \in \bbC$ \cite[pg.\ 360]{GS}. For any $S \in \bbN$, for sufficiently large $\sigma>0$ we have 
	\begin{equation} 
		\langle \sigma \rangle^s = E_S(\sigma) + \sum_{n=0}^S \beta_{n,s}  \sigma^{s-n}
		\label{eq:3g4}
	\end{equation} 
	for some $E_S \in \langle \sigma \rangle^{s-S-1} L^\infty[\sigma_0,\infty)$, for some $\sigma_0>0$. Then, 
	\begin{equation} 
		\calF^{-1} P_{s,j}(\sigma) = \alpha_j \Theta(\sigma) \sigma^{j-1} E_S(\sigma) + \sum_{n=0}^S \alpha_j \beta_{n,s} \Theta(\sigma) \sigma^{s+j-n-1}
		\label{eq:93f}
	\end{equation} 
	for $\sigma>\sigma_0$. Defining $E_S:\bbR\backslash \{0\}\to \bbC$ by \cref{eq:93f} \textit{for all} $\sigma\in \bbR\backslash \{0\}$,  $E_S \in \langle \sigma \rangle^{s-S-1} L^\infty(\bbR) + \sigma^{1-j} L^1_{\mathrm{c}}(\bbR)$. If $S<s+j$ then \cref{eq:93f} holds as an equality between tempered distributions. (For such $S$, $\Theta(\sigma) \sigma^{s+j-n} \in L^1_{\mathrm{loc}}(\bbR_\sigma)$ for all $n=0,\ldots,S$, so the right-hand side of \cref{eq:93f} actually makes sense as a tempered distribution.) We take $S = \lfloor s+j  \rfloor$, which is strictly less than $s+j$ if $s+j$ is not an integer. We have  $ \sigma^{j-1} \langle \sigma \rangle^{s-S-1} L^\infty(\bbR) + \sigma^{j-1} L^1_{\mathrm{c}}(\bbR) \subset L^1(\bbR)$, so $\Theta(\sigma) \sigma^{j-1} E_s(\sigma) \in L^1(\bbR)$. 
	Taking the Fourier transform, 
	\cref{eq:l46} holds for 
	\begin{equation} 
		E = \alpha_j \calF [\Theta(\sigma)\sigma^{j-1} E_S(\sigma)] \in C^0(\bbR),
	\end{equation} 
	with $c_\ell = 0$ for $\ell >j+s$. (And we see that the leading order coefficient in the main sum, $c_0$, is nonzero.)
	
	From \cref{eq:l46}, we conclude that $P_{s,j} \in \calS'(\bbR)\cap L^p_{\mathrm{loc}}(\bbR)$ if and only if $j+s \leq 0$ or if $\tau^{-j-s} \in L^p_{\mathrm{loc}}(\bbR)$, meaning that 
	\begin{equation}
		\int_0^1 \frac{1}{\tau^{pj+ps}} \dd \tau < \infty. 
	\end{equation}
	This holds if and only if $pj+ps < 1$. (If $j+s \leq 0$, this holds.) This is equivalent to $s < p^{-1} - j$. This completes the proof of the proposition, except when $s+j$ is an integer, which we handle by reduction to the generic case:
	\begin{itemize}
		\item If $s<p^{-1} - j$, then we can find some $s_0$ with $s<s_0 < p^{-1} -j$ and $s_0 + j$ is not an integer. The argument above therefore shows that $P_{s_0,j} \in \calS'(\bbR)\cap L^p_{\mathrm{loc}}(\bbR)$. We have $P_{s,j} = \langle D \rangle^{s - s_0} P_{s_0,j}$. Since $\langle D \rangle^{s-s_0}: \calS'(\bbR)\to \calS'(\bbR)$ is a Fourier multiplier (and a Kohn-Nirenberg $\Psi$DO) of negative order, the subspace $\calS'(\bbR)\cap L^p_{\mathrm{loc}}(\bbR) \subset \calS'(\bbR)$ is closed under it, so $P_{s,j} \in \calS'(\bbR)\cap  L^p_{\mathrm{loc}}(\bbR) $. 
		\item If $s> p^{-1} - j$, then we can find some $s_0$ with $s>s_0 > p^{-1} - j$ and $s_0 + j$ is not an integer. If $P_{s,j}$ were in $\calS'(\bbR)\cap L^p_{\mathrm{loc}}(\bbR)$, then $P_{s_0,j} = \langle D \rangle^{-(s-s_0)} P_{s,j}$ would be too. But the argument above shows that $P_{s_0,j} \notin L^p_{\mathrm{loc}}(\bbR)$. 
		\item 
		The one remaining case of the proposition is $s = p^{-1} - j$ and $s+j \in \bbZ$. Since $p\in [1,\infty)$, this can only hold if $p=1$. If $P_{s,j}$ were in $\calS'(\bbR)\cap L^1_{\mathrm{loc}}(\bbR)$, then, for any Schwartz $\chi$, $(\calF^{-1} \chi)P_{s,j} \in L^1(\bbR)$, which implies that 
		\begin{equation} 
			\chi(\sigma) * \langle \sigma \rangle^s \Theta(\sigma) \sigma^{j-1} \in C^0(\bbR)
		\end{equation} 
		and satisfies
		\begin{equation} 
			\lim_{\sigma\to\infty} \chi(\sigma) * \langle \sigma \rangle^s \Theta(\sigma) \sigma^{j-1} = 0.
		\end{equation} 
		But, if $\int_{-\infty}^{+\infty} \chi(\sigma)\dd \sigma \neq 0$ then this is false, as $\lim_{\sigma\to\infty} \chi(\sigma) * \langle \sigma \rangle^s \Theta(\sigma) \sigma^{j-1} = \int_{-\infty}^{+\infty}\chi(\sigma) \dd \sigma$.
	\end{itemize}
\end{proof}

We deduce:
\begin{propositionp}
	\label{prop:LP_cor}
	Suppose that $f \in \calS'(\bbR)$ is a tempered distribution with singular support at a single point $T\in \bbR$, and suppose that 
	\begin{equation}
		f(\tau) = \sum_{ j \in \calJ} \frac{\alpha_j}{(\tau-T\pm i0)^j} \bmod C^\infty 
	\end{equation}
	in some neighborhood of $T$, 
	for some finite subset $\calJ\subset [1,\infty)$ and some $\alpha_j \in \bbC$. Then, for  $s\in \bbR$ and $p\in [1,\infty)$, $f\in \langle D \rangle^{-s} L^p_{\mathrm{loc}}(\bbR)$ if and only if $\operatorname{max} \calJ <p^{-1} - s$. 
\end{propositionp}

\begin{proposition}
	\label{prop:Fourier_comp} 
	\Cref{eq:misc_000} holds, and the Fourier transform $\calF N_{3;\Lambda,\bfA}(\tau) \in \calS'(\bbR_\tau)$  of $N_{3;\Lambda,\bfA}$ is given by 
	\begin{equation}
		\calF\Big(N_{3;\Lambda,\bfA}(\Sigma) - \frac{4\pi}{3} \frac{\Sigma^3}{\operatorname{Covol}(\Lambda)} \Theta(\Sigma)\Big)(\tau) =  \frac{8\pi}{\operatorname{Covol}(\Lambda)}  \sum_{\bfk\in \Lambda^*\backslash \{\bf0\}} \frac{\cos(2\pi \bfk\cdot \bfA)}{((\tau -i0)^2 - 4\pi^2 |\bfk|^2)^2},
		\label{eq:3g3}
	\end{equation}
	where the sum on the right-hand side is (unconditionally) convergent in $\calS'(\bbR_\tau)$.
\end{proposition}
\begin{proof} 
	Since $\calF : \calS'(\bbR)\to \calS'(\bbR)$ is continuous, the Fourier transform of the sum on the right-hand side of \Cref{prop:Delta_full} is, up to a factor of $1/\operatorname{Covol}(\Lambda)$,  
	\begin{multline}
		\calF \sum_{\bfk \in \Lambda^*\backslash \{\bf0\} } \frac{2\pi}{|\bfk|^{1/2}} \cos(2\pi \bfk\cdot \bfA) \Theta(\Sigma) \Sigma^{3/2} J_{1/2}(2\pi |\bfk| \Sigma)  \\ = \sum_{\bfk\in \Lambda^* \backslash \{\bf0\}} \frac{2\pi }{|\bfk|^{1/2}} \cos(2\pi \bfk\cdot \bfA) \calF [\Theta(\Sigma)\Sigma^{3/2} J_{1/2}(2\pi |\bfk| \Sigma)]. 
	\end{multline}
	By \Cref{prop:HWT_comp}, the right-hand side is $
	\operatorname{Covol}(\Lambda)^{-1} 8\pi i \tau\sum_{\bfk\in \Lambda^* \backslash \{\bf0\}}   \cos(2\pi \bfk\cdot \bfA) ((\tau-i0)^2-4\pi^2 |\bfk|^2)^{-2}$.
	
	Combining this with standard formula for $\calF [\Theta(\Sigma) \Sigma^2]$ \cite{GS}, we get \cref{eq:misc_000}. On the other hand, integrating, we get \cref{eq:3g3}, \cref{eq:misc_001}.
\end{proof}

Consider now the case $\Lambda=\bbZ^3,\bfA=0$. 
The series, 
\begin{equation}
	\sum_{n=1}^\infty \frac{r_3(n)}{((\tau -i0)^2 - 4\pi^2 n)^2}, 
	\label{eq:n95}
\end{equation}
where $r_3(n)=\{(n_1,n_2,n_3)\in \bbZ^3:n_1^2+n_2^2+n_3^2=n\}$,
on the right-hand side of \cref{eq:3g3} is convergent in the topology of tempered distributions. Moreover, the series converges pointwise for all $\smash{\tau \notin 2\pi \sqrt{\bbN^+}}$. In fact, it converges uniformly on compact neighborhoods disjoint from $\smash{2\pi \sqrt{\bbN^+}}$ (and in fact the same holds for derivatives). At $\tau \in \smash{2\pi \sqrt{\bbN^+}}$, only one term in \cref{eq:n95} is singular, and the sum of the rest is uniformly convergent in some neighborhood. Consequently, the singularity of the right-hand side of \cref{eq:n95} at $\tau = 2 \pi \sqrt{n}$ is no worse than the singularity of the term $((\tau-i0)^2 - 4\pi^2 n)^{-2}$ at $\tau = 2 \pi \sqrt{n}$, and -- assuming that $r_3(n) \neq 0$ -- it is precisely as bad. So, we expect the local Sobolev regularity of $\calF(N_3(\Sigma) - (4/3)\pi \Sigma^3 \Theta(\Sigma))(\tau)$ to be explicitly analyzable by expanding $((\tau-i0)^2 - 4\pi^2 n)^{-2}$ in Laurent series around $\tau = 2 \pi \sqrt{n}$. The global estimates will require an additional argument.

When computing the Fourier transforms of $N_{3,k;\Lambda,\bfA}$ for $k\geq 1$, we multiply by a nonzero power of $1/\tau$. For $\tau\neq 0$, the upshot of the previous paragraph remains unchanged, but this division produces a singularity at $\tau = 0$ in each term in \cref{eq:n95}. Let $\tau_2 = \tau^2$. For $r>0$, 
\begin{equation}
	\frac{\partial^k}{\partial \tau_2^k}\frac{1}{(\tau_2 - 4\pi^2 r^2)^2} \Big|_{\tau_2 = 0} = \frac{ (k+1)!}{ (4\pi^2 r^2)^{k+2}} = \frac{(k+1)!}{(2\pi)^{2k+4}} \frac{1}{r^{2k+4}}. 
\end{equation}
The $K$th order Taylor series of $(\tau_2-4\pi^2 r^2)^{-2}$ in $\tau_2$ around $\tau_2 = 0$ is therefore 
\begin{equation}
	\sum_{k=0}^K \frac{\tau_2^k}{k!}	\frac{\partial^k}{\partial \tau_2^k}\frac{1}{(\tau_2 - 4\pi^2 r^2)^2} \Big|_{\tau_2 = 0} = \sum_{k=0}^K \frac{k+1}{(2\pi)^{2k+4}} \frac{\tau_2^k}{r^{2k+4}},
\end{equation}
which we consider as an even polynomial of $\tau$. 

Consider, for each $L\in \bbN$, $K\in \bbN\cup \{-1\}$, the tempered distribution
\begin{equation}
	F_{L,K,r}(\tau) =  \Big[\frac{1}{(\tau-i0)^{L}}\frac{1}{((\tau-i0)^2 - 4\pi^2 r^2)^2} -  \sum_{k=0}^K \frac{k+1}{(2\pi)^{2k+4}} \frac{1}{r^{2k+4}} \frac{1}{(\tau-i0)^{L-2k}} \Big],
	\label{eq:Fdef}
\end{equation} 
which we will estimate.
\begin{itemize}
	\item Via Taylor's theorem, $F_{L,K,r}$ is a function for $\tau$ in some punctured neighborhood of the origin, of size $O(\tau^{2K+2-L})$ as $\tau\to 0$. We will choose $K$ so that the singularity of $F_{L,K,r}(\tau)$ at $\tau = 0$ is no greater than the singularity of $(\tau_2 - 4\pi^2 r^2)^{-2}$ at $\tau = \pm 2\pi r^2$, that is second order. In other words, we should take $2K+2-L \geq -2$, i.e.\ $K\geq \lceil L/2 \rceil - 2$.
	\item As $\tau \to \pm \infty$, on the other hand, $F_{L,K,r}(\tau) = O(|\tau|^{2K-L})$. In order for $\langle \tau \rangle^{2K-L}$ to be integrable, we should have $2K-L < -1$, i.e.\ $K<(L-1)/2$. 
\end{itemize}
In particular, $K = \lceil L/2 \rceil - 2$ satisfies our requirements.

\begin{proposition}
	\label{prop:inclusion}
	For any $L\in \bbN$, 
	$\lVert \calF F_{L,K,1} \rVert_{\langle \Sigma \rangle L^\infty(\bbR_\Sigma)} < \infty$ if $K= \lceil L/2 \rceil - 2$.  
\end{proposition}
\begin{proof}
	We choose $\chi_1,\chi_2 \in C_{\mathrm{c}}^\infty(\bbR)$ and $\chi_3 \in \calS(\bbR)$ such that $\chi_1$ is supported in a neighborhood of the origin away from $2\pi$, $\chi_2$ is supported in a neighborhood of $2\pi $ away from the origin, $\chi_3$ is identically zero in $[-1,8]$, and $\chi_1+\chi_2+\chi_3 = 1$. (Such functions exist.) Then 
	\begin{equation}
		\lVert \calF F_{L,K,1} \rVert_{\langle \Sigma \rangle L^\infty(\bbR_\Sigma)} \leq \sum_{j=1}^3 \lVert \calF \chi_j F_{L,K,1} \rVert_{\langle \Sigma \rangle L^\infty(\bbR_\Sigma)} . 
		\label{eq:5gh}
	\end{equation}
	The function $\chi_3 F_{L,K,1}$ is smooth and absolutely integrable, so 
	\begin{equation} 
		\lVert \calF \chi_3 F_{L,K,1} \rVert_{\langle \Sigma \rangle L^\infty(\bbR_\Sigma)}<\infty.
		\label{eq:chiF3}
	\end{equation}  
	On  the other hand, $\chi_1 F_{L,K,1} = E(\tau) + \alpha \chi_1(\tau) (\tau -i0)^{-2} +  \beta \chi_1(\tau) (\tau - i0)^{-1}$ for some smooth, absolutely integrable $E$ and some $\alpha,\beta \in \bbC$. Then $\calF E$ is continuous and uniformly bounded, and $\calF (\chi_1 (\tau-i0)^{-2})(\Sigma) = \calF \chi_1 * \calF ((\tau-i0)^{-2})(\Sigma)$ is the convolution of a Schwartz function with an element of $\langle \Sigma \rangle L^\infty (\bbR_\Sigma)$ and is therefore in $\langle \Sigma \rangle L^\infty(\bbR_\Sigma)$ as well. An analogous argument applies to $\calF (\chi_1(\tau) (\tau-i0)^{-1})(\Sigma)$. We conclude that  
	\begin{equation} 
		\lVert \calF \chi_1 F_{L,K,1} \rVert_{\langle \Sigma \rangle L^\infty(\bbR_\Sigma)}<\infty.
		\label{eq:chiF1}
	\end{equation} 
	A similar argument applies to the second term in \cref{eq:5gh}, so 
	\begin{equation} 
		\lVert \calF \chi_2 F_{L,K,1} \rVert_{\langle \Sigma \rangle L^\infty(\bbR_\Sigma)}<\infty.
		\label{eq:chiF2}
	\end{equation} 
	Combining the estimates \cref{eq:chiF3}, \cref{eq:chiF1}, and \cref{eq:chiF2} with \cref{eq:5gh}, we get the result	$\lVert \calF F_{L,K,1} \rVert_{\langle \Sigma \rangle L^\infty(\bbR_\Sigma)} < \infty$.
\end{proof}

\begin{proposition}
	\label{prop:Fourier_bound}
	If $L\in \{2,3,4,\cdots\}$ and  $K= \lceil L/2 \rceil -2$, 
	\begin{equation}
		\sum_{\bfk\in \Lambda^* \backslash \{\bf0\}}  \lVert \calF F_{L,K,|\bfk|} \rVert_{\langle \Sigma \rangle L^\infty(\bbR_\Sigma)} < \infty. 
	\end{equation}
\end{proposition}

\begin{proof} 
	First note that $F_{L,K,|\bfk|}(\tau) = |\bfk|^{-4-L} F_{L,K,1}(\tau/|\bfk|)$. (Indeed, we can replace the $\tau - i0$ in \cref{eq:Fdef} with $|\bfk|(\tau/|\bfk|-i0)$ and then factor out the factor of $|\bfk|$.)  Second, note that $\calF [F_{L,K,1}(\tau/|\bfk|)](\sigma) = |\bfk| \calF F_{L,K,1}(|\bfk| \sigma)$. Third, 
	\begin{align} 
		\lVert \calF F_{L,K,1}(|\bfk|\sigma) \rVert_{\langle \sigma \rangle L^\infty(\bbR_\sigma)} &= \lVert \langle \sigma \rangle^{-1} \calF F_{L,K,1}(|\bfk| \sigma) \rVert_{L^\infty(\bbR_\sigma)} \\ &=  \lVert \langle |\bfk|^{-1} \sigma \rangle^{-1} \calF F_{L,K,1}( \sigma) \rVert_{L^\infty(\bbR_\sigma)}.
	\end{align} 
	Fourth, $\langle |\bfk|^{-1} \sigma \rangle^{-1} \leq  |\bfk| \langle \sigma \rangle^{-1}$. Therefore, $\lVert \calF F_{L,K,|\bfk|} \rVert_{\langle \sigma \rangle L^\infty(\bbR_\sigma)} \leq  |\bfk|^{-2-L} \lVert \calF F_{L,K,1} \rVert_{\langle \sigma \rangle L^\infty(\bbR_\sigma)}$.
	It follows that
	\begin{align}
		\begin{split} 
		\sum_{\bfk\in \Lambda^* \backslash \{\bf0\}}^\infty\lVert \calF F_{L,K,|\bfk|} \rVert_{\langle \sigma \rangle L^\infty(\bbR_\sigma)} &\leq   \lVert \calF F_{L,K,1} \rVert_{\langle \sigma \rangle L^\infty(\bbR_\sigma)} \sum_{\bfk \in \Lambda^* \backslash \{\bf0\}} \frac{1}{|\bfk|^{2+L}}.
		\end{split}
	\end{align}
	This will be finite if (and only if) $L>1$ (and since $L$ is an integer, $L\geq 2$) and $\lVert \calF F_{L,K,1} \rVert_{\langle \sigma \rangle L^\infty(\bbR_\sigma)} <\infty$. By \Cref{prop:inclusion}, this holds for the given $K$. 
\end{proof}

Since
\begin{equation} 
	N_{3,k;\Lambda,\bfA}'(\sigma) = N_{3,k-1;\Lambda,\bfA}(\sigma)
\end{equation} 
holds as an identity of tempered distributions,
\begin{equation} 
	\calF N_{3,k-1;\Lambda,\bfA}(\tau) = \calF [N_{3,k;\Lambda,\bfA}'](\tau) = i \tau \calF N_{3,k;\Lambda,\bfA}(\tau).
\end{equation} 
We inductively conclude (using \Cref{prop:Fourier_integration} repeatedly) that 
\begin{multline} 
	\calF N_{3,k;\Lambda,\bfA}(\tau) = \Big( - \frac{i}{\tau - i0} \Big)^k \calF N_{3,0;\Lambda,\bfA}(\tau) 
	= \Big( - \frac{i}{\tau-i0} \Big)^{k+1} \calF \sum_{\bmSigma \in \Lambda} \delta(\Sigma - |\bmSigma+\bfA|) \\ = \Big( - \frac{i}{\tau-i0} \Big)^{k+1} \frac{1}{\operatorname{Covol}(\Lambda)} \calF \Big[ 4\pi \Sigma^2 \Theta(\Sigma) + \Sigma^{3/2}  \sum_{\bfk\in \Lambda^*\backslash \{\bf0\}} \frac{2\pi }{|\bfk|^{1/2}}  \cos(2\pi \bfk\cdot \bfA) \Theta(\Sigma)  J_{1/2}(2\pi |\bfk| \Sigma)  \Big].
\end{multline}
Consequently, letting $\bar{N}_{3,k;\Lambda,\bfA}(\Sigma) = N_{3,k;\Lambda,\bfA}(\Sigma)- 8\pi \operatorname{Covol}(\Lambda)^{-1} \Sigma^{3+k} \Theta(\Sigma)/(3+k)!$,
we have  
\begin{align}
	\begin{split} 
	\calF\bar{N}_{3,k;\Lambda,\bfA}(\tau)&= \Big( - \frac{i}{\tau-i0} \Big)^{k+1} \frac{1}{\operatorname{Covol}(\Lambda)} \calF \Big[ \Sigma^{3/2}  \sum_{\bfk\in \Lambda^* \backslash \{\bf0\}}\frac{2\pi}{|\bfk|^{1/2}} \cos(2\pi \bfk\cdot \bfA)  \Theta(\Sigma)  J_{1/2}(2\pi |\bfk| \Sigma)  \Big] \\
	&= \Big( - \frac{i}{\tau-i0} \Big)^{k+1} \frac{1}{\operatorname{Covol}(\Lambda)} \Big[ 8\pi i \tau \sum_{\bfk\in \Lambda^* \backslash \{\bf0\}}  \frac{\cos(2\pi \bfk\cdot \bfA)}{((\tau-i0)^2-4\pi^2 |\bfk|^2 )^2} \Big] \\
	&= \Big( - \frac{i}{\tau-i0} \Big)^{k} \frac{ 8\pi}{\operatorname{Covol}(\Lambda)} \Big[  \sum_{\bfk\in \Lambda^* \backslash \{\bf0\}}  \frac{\cos(2\pi \bfk\cdot \bfA)}{((\tau-i0)^2-4\pi^2 |\bfk|^2)^2} \Big].
	\end{split} 
	\label{eq:misc_94h}
\end{align}
For any $K\in \bbN$, we can rewrite \cref{eq:misc_94h} as follows (replacing the dummy variable `$k$' with $m \in \bbN$ to avoid conflict):
\begin{equation}
	\calF\bar{N}_{3,m;\Lambda,\bfA}(\tau) = \frac{8\pi (-i)^m}{\operatorname{Covol}(\Lambda)} \sum_{\bfk\in \Lambda^* \backslash \{\bf0\}} \cos(2\pi \bfk\cdot \bfA)\Big[ F_{m,K,|\bfk|}(\tau)+ \frac{1}{(\tau-i0)^m}\sum_{k=0}^K \frac{k+1}{(2\pi)^{2k+4}} \frac{\tau_2^k}{|\bfk|^{2k+4}}   \Big].
\end{equation}
Rearranging, 
\begin{multline}
	\calF\bar{N}_{3,m;\Lambda,\bfA}(\tau) = \frac{8\pi}{\operatorname{Covol}(\Lambda)} (-i)^m \Big[\sum_{\bfk\in \Lambda^* \backslash \{\bf0\}}  \cos(2\pi \bfk\cdot \bfA) F_{m,K,|\bfk|}(\tau)\Big] \\+ \frac{8 \pi}{\operatorname{Covol}(\Lambda)} \Big( - \frac{i}{\tau-i0}\Big)^m\sum_{k=0}^K \frac{k+1}{(2\pi)^{2k+4}} \tau_2^k \sum_{\bfk\in \Lambda^* \backslash \{\bf0\}} \frac{\cos(2\pi \bfk\cdot \bfA)}{|\bfk|^{2k+4}}.
	\label{eq:misc_111}
\end{multline}
(Note that $\sum_{\bfk\in \Lambda^* \backslash \{\bf0\}} |\bfk|^{-2k-4}$ is absolutely convergent for each $k\geq 0$.)

We now take $m\geq 2$ and $K = \lceil m/2 \rceil - 2$, so that $\sum_{\bfk \in \Lambda^* \backslash \{\bf0\}} \lVert \calF^{-1} F _{m,K,|\bfk|} \rVert_{\langle \Sigma \rangle L^\infty(\bbR_\Sigma)} < \infty$, per the conclusion of \Cref{prop:Fourier_bound}. 

Recall the following computation --- cf.\ \cite[pg.\ 360]{GS}: for each $\ell \in \bbR^+$, \begin{equation} 
	\calF^{-1}[(\tau-i0)^{-\ell}](\sigma) =  (+i)^{\ell} \Theta(\sigma) \sigma^{\ell-1} /  (\ell-1)!.
\end{equation} 
Applying $\calF^{-1}$ to \cref{eq:misc_111} and using the fact that $\sum_{\bfk\in \Lambda^* \backslash \{\bf0\}} \cos(2\pi \bfk\cdot \bfA) F_{m,K,|\bfk|}(\tau)$ is (unconditionally) convergent in $\calS'(\bbR_\tau)$, 
\begin{align}
	\bar{N}_{3,m;\Lambda,\bfA}(\Sigma) &= O( \Sigma ) + \frac{8 \pi (-i)^m}{\operatorname{Covol}(\Lambda)} \sum_{k=0}^K \frac{k+1}{(2\pi)^{2k+4}} \Big( \sum_{\bfk\in \Lambda^* \backslash \{\bf0\}} \frac{\cos(2\pi \bfk\cdot \bfA)}{|\bfk|^{2k+4}} \Big) \calF^{-1}[(\tau-i0)^{2k-m}](\Sigma) \\
	&= O(\Sigma ) + \frac{8 \pi}{\operatorname{Covol}(\Lambda)}  \sum_{k=0}^K \frac{k+1}{(2\pi)^{2k+4}} (-1)^{k} \Big( \sum_{\bfk\in \Lambda^* \backslash \{\bf0\}}\frac{\cos(2\pi \bfk\cdot \bfA)}{|\bfk|^{2k+4}} \Big)   \frac{\Sigma^{m-2k-1}}{(m-2k-1)!}. 
\end{align}
We conclude -- independently of \S\ref{sec:exp},\ref{sec:as_part_1} -- that \cref{eq:iterated_asymptotics} holds, with $C_{j;\Lambda,\bfA} = 0$ if $j \in \bbN$ is odd and with 
\begin{equation} 
	C_{j;\Lambda,\bfA} =   \frac{(2j+4)}{\operatorname{Covol}(\Lambda)  (2\pi)^{j+3}} (-1)^{j/2} \sum_{\bfk\in \Lambda^* \backslash \{\bf0\}} \frac{1}{|\bfk|^{4+j}} \cos(2\pi \bfk\cdot \bfA) 
\end{equation} 
if $j$ is even.

\begin{proposition}
	For all $k\in \bbN$, there does not exist an $\epsilon>0$ such that, for some polynomial $Z(\sigma) = Z_{\Lambda,\bfA}(\sigma) \in \bbC[\sigma]$, we have $N_{3,k;\Lambda,\bfA}(\Sigma)  - Z(\Sigma) = O(\Sigma^{1-\epsilon})$
	as $\Sigma\to\infty$. 
\end{proposition}
\begin{proof}
	If $N_{3,k;\Lambda,\bfA}(\Sigma)  - Z(\Sigma) = O(\Sigma^{1-\epsilon})$ as $\Sigma\to\infty$, then $N_{3,k;\Lambda,\bfA}(\Sigma)  - \Theta(\Sigma) Z(\Sigma) \in \langle \Sigma \rangle^{3/2} L^2(\bbR_\Sigma)$. 
	Taking the Fourier transform (and using the fact that $\calF[\Theta Z](\tau)$ is singular only at $\tau=0$), we deduce that 
	\begin{equation} 
		\calF N_{3;\Lambda,\bfA}(\tau) \in \langle D \rangle^{3/2}(\calS'(\bbR)\cap L^2_{\mathrm{loc}}(\bbR_\tau\backslash \{0\})).
	\end{equation} 
	Via \Cref{prop:Fourier_comp}, expanding the $\bfk$th term in \cref{eq:3g3} in Taylor series around $\tau = 2\pi |\bfk|$ and checking that the sum of the rest is smooth in a neighborhood of this point, and applying \Cref{prop:LP_cor} with $j=2,s=-3/2,p=2$,  this is not the case.
\end{proof}

This completes the proof of \Cref{thm:main4}. 

\appendix

\section{Integration in $\dot{\calS}'(\bbR^{\geq 0})$}
\label{sec:app}

Recall Schwartz's definition of differentiation of distributions, $\partial:\calS'(\bbR)\to \calS'(\bbR)$ (restricting attention to tempered distributions, for simplicity), defined by
\begin{equation}
	(\partial u)(\chi) = - u(\partial \chi), 
\end{equation}
for any $u\in \calS'(\bbR)$, where $\chi \in \calS(\bbR)$.  Recall or observe that $\smash{\ker_{\calS'(\bbR)}(\partial) = \{u\in \calS'(\bbR): \partial u = 0\}}$ consists precisely of constant functions. 

Restricting $\partial$ to the (topological and linear) subspace $\dot{\calS}'(\bbR^{\geq 0})\subset \calS'(\bbR)$, we get a continuous and linear  map 
\begin{equation} 
	\partial:\dot{\calS}'(\bbR^{\geq 0})\to \dot{\calS}'(\bbR^{\geq 0}),
	\label{eq:partial_res}
\end{equation}
and, since the only constant function in $\smash{\dot{\calS}'(\bbR^{\geq 0})}$ is identically zero, this is \emph{injective}, unlike $\partial :\calS'(\bbR)\to \calS'(\bbR)$. Likewise, since $\partial \bbC[\partial] \delta \subset \bbC[\partial]\delta$, $\partial$ induces a continuous linear map
\begin{equation} 
	\partial: \dot{\calS}'(\bbR^{\geq 0})/ \bbC[\partial] \delta \to \dot{\calS}'(\bbR^{\geq 0})/\bbC[\partial] \delta
\end{equation}
on the quotient LCTVS $\smash{\dot{\calS}'(\bbR^{\geq 0})/ \bbC[\partial] \delta}$. This LCTVS is (per our conventions) canonically identified with $\calS'(\bbR^{\geq 0})$, so we can consider the map as a continuous, linear map $\smash{\partial:\calS'(\bbR^{\geq 0})\to \calS'(\bbR^{\geq 0})}$.
\begin{propositionp}
	\label{prop:mod_ker}
	$\ker_{\calS'(\bbR^{\geq 0})}(\partial) = \{u\in \calS'(\bbR^{\geq 0}): \partial u=0\} = \{c \Theta \bmod \bbC[\partial] \delta : c\in \bbC\}$.
\end{propositionp}
The significance is that an extendable distribution on $\bbR^{\geq 0}$ is only determined by its derivative up to a scalar multiple of $\Theta$.
\begin{proposition}
	Given any $u \in \dot{\calS}'(\bbR^{\geq 0})$, there exists some $v \in \dot{\calS}'(\bbR^{\geq 0})$ such $\partial v = u$. 
	\label{prop:fundamental}
\end{proposition}
\begin{proof}
	First, given $u\in \dot{\calS}'(\bbR^{\geq 0})$, we define a map $Iu : \partial\calS(\bbR)\to \bbC$ by
	\begin{equation}
		I u (\chi) =  -\int_{-\infty}^{+\infty} u(\Sigma) \Big[ \int_{-\infty}^\Sigma \chi(\sigma) \dd \sigma \Big] \dd \Sigma = -u \Big( \int_{-\infty}^\bullet \chi(\sigma) \dd \sigma  \Big). 
	\end{equation}
	Note that $I u $ is continuous with respect to the subspace topology on $\partial \calS(\bbR)\subset \calS(\bbR)$. 
	Choose any $\chi_0 \in C_{\mathrm{c}}^\infty(\bbR)$ supported in $\bbR^+$ with $\smash{\int_{-\infty}^{+\infty} \chi_0(\sigma) \dd \sigma = 1}$.  Then, any element $\chi \in \calS(\bbR)$ can be written as 
	\begin{equation}
		\chi(\Sigma) = \Big(\chi(\Sigma) - \chi_0(\Sigma) \int_{-\infty}^{+\infty} \chi(\sigma) \dd \sigma \Big) +\chi_0(\Sigma) \int_{-\infty}^{+\infty} \chi(\sigma) \dd \sigma, 
	\end{equation}  
	and the first summand on the right-hand side is in $\partial \calS(\bbR)$. The map $\Pi : \calS(\bbR)\to \partial \calS(\bbR) \subset \calS(\bbR)$ given by $\smash{\chi \mapsto \chi - \chi_0 \int_{-\infty}^{+\infty} \chi(\sigma) \dd \sigma}$ is continuous, so 
	\begin{equation}
		{}^{\chi_0}\!\!\int u : \calS(\bbR)\to \bbC, \qquad \Big[{}^{\chi_0}\!\!\int u\Big] (\chi) = I u (\Pi \chi),
	\end{equation}
	where $\chi \in \calS(\bbR)$,
	defines a continuous map, hence a tempered distribution.  By construction, ${}^{\chi_0}\!\!\int u(\partial \chi) = I u(\partial \chi)  =-u(\chi)$, so $\partial ({}^{\chi_0}\!\!\int) u = u$. 
	
	Depending on the choice of $\chi_0$, ${}^{\chi_0}\!\!\int u \in \calS'(\bbR)$ is not necessarily in $\dot{\calS}'(\bbR^{\geq 0})$. However, $\smash{\partial {}^{\chi_0}\!\!\int u}$ is zero in $(-\infty,0)$, which implies that there exists some constant $c\in \bbC$ such that $v=-c+ {}^{\chi_0}\!\!\int u$ is in $\smash{\dot{\calS}'(\bbR^{\geq 0})}$, and $v$ satisfies $\partial v = u$. 
\end{proof}
Consequently:
\begin{proposition}
	The operator \cref{eq:partial_res} is an automorphism in the category of TVSs. 
	\label{prop:continuity}
\end{proposition}
\begin{proof}
	By the previous proposition, $\partial:\dot{\calS}'(\bbR^{\geq 0})\to \dot{\calS}'(\bbR^{\geq 0})$ is a bijective linear map, hence a linear automorphism, and it is continuous (since it is continuous on $\calS'(\bbR)$). It remains to show that the set-theoretic inverse
	\begin{equation}
		\int = \partial^{-1} : \dot{\calS}'(\bbR^{\geq 0}) \to \dot{\calS}'(\bbR^{\geq 0}), 
		\label{eq:int}
	\end{equation}
	which is linear,
	is continuous. 
	Since $\dot{\calS}'(\bbR^{\geq 0})\subset \calS'(\bbR)$	is endowed with the subspace topology and $\calS'(\bbR)$ is endowed with the usual weak-$*$ topology, this means showing that $\mathrm{eval}_\chi:\dot{\calS}'(\bbR^{\geq 0})\ni u \mapsto \partial^{-1} u (\chi) \in \bbC$ is continuous for each $\chi \in \calS(\bbR)$. 
	
	Letting $c,\chi_0$ be as in the proof of the previous proposition, 
	\begin{equation}
		\partial^{-1} u (\chi) = -c - u \Big( \int_{-\infty}^\bullet \Big[ \chi(\sigma) - \chi_0(\sigma) \int_{-\infty}^{+\infty} \chi(s) \dd s  \Big] \dd \sigma \Big)
		\label{eq:misc_cu2}
	\end{equation}
	for any $u\in \dot{\calS}'(\bbR^{\geq 0})$. Fixing $\chi$, the Schwartz function in the parentheses above is fixed, so the continuity of \cref{eq:misc_cu2} in $u$ is immediate from the definition of the weak-$*$ topology. 
\end{proof}

\begin{proposition}
	\label{prop:Fourier_integration} 
	If $u,v$ are as in \Cref{prop:fundamental}, then $\calF v = - i (\tau - i0)^{-1} \calF u$, where the multiplication on the right-hand side is well-defined (in the sense of H\"ormander \cite[Definition 2.1, Proposition 2.2]{HAP}). 
\end{proposition}
\begin{proof}
	Assuming that $- i (\tau - i0)^{-1} \calF u \in \calD '(\bbR)$ is defined via \cite[Definition 2.1]{Hormander68}, then, since $u = \partial v$, 
	\begin{equation} 
		-i (\tau - i0)^{-1} \calF u = \tau (\tau-i0)^{-1} \calF v = \calF v
		\label{eq:misc_1mj}
	\end{equation} 
	as an element of $\calD '(\bbR)$. Since a tempered distribution is uniquely determined by its restriction to elements of $C_{\mathrm{c}}^\infty(\bbR)$, we can deduce from \cref{eq:misc_1mj} that $-i (\tau - i0)^{-1} \calF u = \calF v$ as tempered distributions.
	
	So, it only needs to be checked that the product $(\tau - i0)^{-1} \calF u$ is well-defined in the sense of H\"ormander, which means that 
	\begin{equation} 
		\bbR^2= T^* \bbR \supset \widetilde{\operatorname{WF}} [(\tau-i0)^{-1}] \cap \operatorname{WF} [\calF u] = \varnothing, 
		\label{eq:misc_lll}
	\end{equation} 
where the tilde denotes (co)fiberwise reflection across the zero section. This follows from the observation that the Fourier transform of any element of $\dot{\calS}'(\bbR^{\geq 0})$ has one-sided wavefront set; likewise, $(\tau-i0)^{-1}$ has a one-sided wavefront set, and our conventions regarding signs in the Fourier transform are such that the two sides agree, so that \cref{eq:misc_lll} holds.
\end{proof}
\section*{Acknowledgements}

\Cref{fig} was made in \texttt{Mathematica} (and typeset using Szabolcs Horv\'at's \texttt{MaTeX} package). Code used to make figure available on request.

E.S.\ gratefully acknowledges the support of a Hertz fellowship. 

E.F.\ thanks the Center for Excellence in Education, the Research Science Institute, and MIT Mathematics for providing him with this research opportunity. E.F.\ is grateful for the guidance of RSI and MIT math staff.

In addition, the authors thank the reviewer for their careful reading.
\printbibliography

\end{document}